\documentclass[11pt, twoside]{amsart}
\usepackage{graphics}
\usepackage{graphicx}
\usepackage{amssymb}
\usepackage{amsmath}
\usepackage{amsthm}
\usepackage{color}
\usepackage{xcolor}
\usepackage{makecell}
\usepackage{pgf,tikz,pgfplots}
\usepackage{mathrsfs}
\usetikzlibrary{arrows}
\usepackage{graphicx}
\usepackage{indentfirst}

\allowdisplaybreaks

\def\rr{{\mathbb R}}
\def\rn{{{\rr}^n}}

\renewcommand\hat{\widehat}

\def\supp{{\rm{\,supp\,}}}

\def\r{\Big}
\def\lf{\Big}

\def\bint{{\ifinner\rlap{\bf\kern.30em--}
\int\else\rlap{\bf\kern.35em--}\int\fi}\ignorespaces}

\def\sbint{{\ifinner\rlap{\bf\kern.32em--}
\hspace{0.078cm}\int\else\rlap{\bf\kern.45em--}\int\fi}\ignorespaces}

\newtheorem{theorem}{Theorem}[section]
\newtheorem{lemma}[theorem]{Lemma}
\newtheorem{conjecture}[theorem]{Conjecture}

\newtheorem{proposition}[theorem]{Proposition}

\theoremstyle{definition}
\newtheorem{remark}[theorem]{Remark}

\newtheorem{definition}[theorem]{Definition}
\numberwithin{equation}{section}

\numberwithin{equation}{section}

\numberwithin{equation}{section}

\begin{document}

\arraycolsep=1pt

\title{The space-time estimates for the Schr\"odinger equation}

\author{Junfeng Li}
\address{School of Mathematical Sciences, Dalian University of Technology, Dalian, LN, 116024, China}

\email{junfengli@dlut.edu.cn}

\author{Changxing Miao}
\address{Institute of Applied Physics and Computational Mathematics, Beijing, 100088, China}
\email{miao\_changxing@iapcm.ac.cn}

\author{Ankang Yu}
\address{School of Mathematical Sciences, Dalian University of Technology, Dalian, LN, 116024, China}
\email{yak@mail.dlut.edu.cn}

\thanks{This work was supported by the National Key R\&D program of China; No. 2022YFA1005700. Junfeng Li was partially supported by NSF of China grant: No. 12071052 and the fundamental research funds for the central universities.} 



\begin{abstract}
In this paper, we studied the space-time estimates for the solution to the Schr\"odinger equation. By polynomial partitioning, induction arguments, bilinear to linear arguments and broad norm estimates, we set up several maximal estimates for the Schr\"odinger equation with high frequency input data. By these maximal estimates we obtain the sharp global space-time estimate when $n=2$ and improve the known results in the critical cases when $n\geq 3$. The maximal estimate for $n=2$ is also used to extend the results of  the local space-time estimates for the solution to the Schr\"odinger equation.
\end{abstract}
	
\maketitle

\vspace{-0.8cm}


\section{Introduction\label{s1}}
Given the free Schr$\ddot{\text{o}}$dinger equation:
$$i\partial_tu+\Delta u=0,\ \  \text{in}\ \  \mathbb {R}^{n+1}$$
with an initial data $f\in\mathcal{S}(\mathbb R^n)$, its solution can be written as
$$u(t,x):=e^{it\Delta}f(x)=\int_{\mathbb{R}^n}\hat{f}(\xi)e^{2\pi i (x\cdot\xi-2\pi t|\xi|^{2})}d\xi,$$
where $\hat{f}(\xi):=\int_{\mathbb{R}^n}f(x)e^{-2\pi ix\cdot\xi}dx$. A nature problem is to understand the solution $u(t,x)$ with an initial data $f\in H^s(\mathbb R^n)$, the Sobolev space. The well-known Srichartz estimate is a typical estimate. It is a time-space type estimate which weight the solution in $L^r_tL^p_x(\mathbb R^{n+1})$. Such a estimate is  crucial in the study of the linear and nonlinear Schr\"odinger equations. By the end of the last century, it was fully established. \begin{theorem}\cite{kt98,s77}.\label{t1.0}
	Let $r\geq 2$, $2\leq p<\infty$ and $\frac{n}{p}+\frac{2}{r}\leq\frac{n}{2}$. Then
\begin{equation}\label{Strichartz}
	\|e^{it\Delta}f\|_{L_{t}^{r}(\mathbb{R}, L_{x}^{p}(\mathbb{R}^n) )}\leq C\|f\|_{\dot{H}^{s}({\mathbb{R}^n})},  \  \  \  \    s=\frac{n}{2}-\frac{n}{p}-\frac{2}{r}.
\end{equation}
\end{theorem}
In this paper, we first consider the space-time counter part of the Strichartz estimate. That is weight the function $u(t,x)$ in the space $L^p_xL^r_t(\mathbb R^{n+1})$.
In \cite{p02} Planchon formulated the following space-time conjecture. 
\begin{conjecture}\label{c1.1}
	Let  $2\leq r,p<\infty$ and $\frac{n+1}{p}+\frac{1}{r}\leq\frac{n}{2}$. Then
	\begin{equation}\label{S_{p,r}}
		\|e^{it\Delta}f\|_{L_{x}^{p}L^r_t(\mathbb{R}^{n+1})}\leq C\|f\|_{\dot{H}^{s}({\mathbb{R}^n})},  \  \  \  \    s=\frac{n}{2}-\frac{n}{p}-\frac{2}{r}.
	\end{equation}
\end{conjecture} 
Both inequalities \eqref{Strichartz} and \eqref{S_{p,r}} are the anisotropic generalization of the adjoint estimate of the Fourier restriction estimate. They satisfy some common symmetries such as translation invariance and scaling invariance. The former allows us to do translations over $x$ and $t$, and the latter predicts the relation $s=\frac{n}{2}-\frac{n}{p}-\frac{2}{r}$. The necessary conditions for the index pairs $(p,r)$ can be verified by Knapp's example.  

The estimates \eqref{Strichartz} and \eqref{S_{p,r}} behave differently under Galilean transforms.  The Strichartz estimate \eqref{Strichartz} is Galilean invariant. By which we can make translation in the frequency space.  While the space-time estimate \eqref{S_{p,r}} is not Galilean invariant which makes space-time estimate more difficult. And the estimate would depend on both the size and the position of the support of $\hat{f}$.

\subsection{Maximal estimates} The space-time estimate and its variant has close relation with several important problems in  harmonic analysis. One of the endpoint index  pair $(\frac{2(n+1)}{n},\infty)$ of its critical case appears naturally in the study of the well-known Carleson's problem. Here and after, the index pair $(p,r)$ is critical if $\frac{n+1}{p}+\frac{1}{r}=\frac{n}{2}$. And the pair in the set \{$(p,r):\frac{n+1}{p}+\frac{1}{r}<\frac{n}{2}$\} will be called subcritical. In  \cite{C79} Carleson formulated the problem to find the the infimum regularity $s_c$ such that
$$\lim_{t\to0}e^{it\Delta}f(x)=f(x)\quad a.e.\,x\in\mathbb R^n,\,\forall f\in H^s(\mathbb R^n).$$
In the last few decades, there were abundant literature such as \cite{l06,clv12,dgl17,dglz18,dz18,lr19} on this problem. So far, Carleson's problem is well understood, and the critical index is  $s_c=\frac{n}{2(n+1)}$ for $n\geq1$. The one dimensional case follows from the results of Carleson \cite{C79} and Dahlberg and Kenig \cite{dk82}. For the high dimensional case, the necessary part follows from the counterexample formulated by Bourgain in \cite{b16}. To achieve the positive result, Du, Guth, and Li \cite{dgl17} set up the following $\epsilon$ lost estimate
\begin{equation}\label{DGL}
		\|e^{it\Delta}f\|_{L_x^3L_t^{\infty}(B^3(0,R))}\lesssim_{\epsilon} R^{\epsilon}\|f\|_{L^2},
	\end{equation}
	for $\text{supp} \hat{f}\subset \mathbb{B}^2(0,1)$ when $n=2$. 
Motivated by the two dimensional case result,  one may expect to set up the $\epsilon$ lost estimate 
\begin{equation}\label{local maximal estimate}
	\|e^{it\Delta}f(x)\|_{L_x^{\frac{2(n+1)}{n}}L^\infty_t(\mathbb{B}^{n+1}(0,R))}\lesssim_\epsilon R^\epsilon \|f\|_{L^2(\mathbb R^n)}
\end{equation}
for $\text{supp} \hat{f}\subset \mathbb{B}^n(0,1)$. Which corresponding to endpoint index of Conjecture \ref{c1.1}. With \eqref{local maximal estimate}, combining the $\epsilon$ removal argument in \cite{ckl22}, one can obtain \eqref{S_{p,r}} for the whole subcritical case.  

 While Du, Kim, Wang, and Zhang \cite{dkwz19} showed that the estimate \eqref{local maximal estimate} fails for $n\geq3$. As a replacement, Du and Zhang \cite{dz18} proved
\begin{equation}\label{e1.1}
	\|\sup\limits_{0<t\leq R}|e^{it\Delta}f|\|_{L^{2}(\mathbb{B}^{n}(0,R))}\le R^{\frac{n}{2(n+1)}+\epsilon}\|f\|_{L^{2}(\mathbb{R}^{n})},\ \ \ \ \forall \epsilon>0,
\end{equation}
for $\text{supp} \hat{f}\subset \mathbb{B}^n(0,1)$.  It is a meaningful problem to find the sharp range of $p$ such that 
\begin{equation}\label{e1.5}
	\|\sup_{0<t\leq R}|e^{it\Delta}f(x)|\|_{L^{p}(\mathbb{B}^n(0,R))}\lesssim_\epsilon R^\epsilon \|f\|_{L^2(\mathbb R^n)}
\end{equation}
for $\text{supp} \hat{f}\subset \mathbb{B}^n(0,1)$. Recently,  Wu \cite{w20}, Cao, Miao, and Wang \cite{cmw21} have maken great improvement on this problem. 

Since the local estimate \eqref{local maximal estimate} does not hold therefore the corresponding global estimate cannot hold either. Thus, Conjecture \ref{c1.1} becomes a meaningful replacement of the maximal estimate. To achieve this aim, we set up several maximal estimates with high frequency input functions. Which also reflect the fact that the frequency of the input function affects the space-time estimates. 
 Given $z_0=(x_0,t_0)$, $N\geq1$ and a usually very large number $R\geq1$, $Q_R(x_0)\subset\mathbb R^n$ denotes a cube centered at $x_0$ with side length $R$.  A parallelepiped with direction $(4\pi Ne_1,1)$ and center $z_0$ is 
 $$P_R(z_0):=\{(x,t)\in \mathbb{R}^n\times [-R+t_0,R+t_0]:x-4\pi tNe_1\in Q_R(x_0)\},$$ here $e_1=(1,0,\dots,0)$. The parameter $R$ will be called as its size of $P_R(z_0)$. Let $P_R=P_R(z_0)$ with $z_0=(0,0)$. We now state our main results in this part.

\begin{theorem}[High frequency input endpoint estimate for $n=2$]\label{t2.1}
	Let $\hat{f}$ be supported in $ \mathbb{B}^2(Ne_1,1)$ with $N\geq 1$. For any $\epsilon>0$, we have
	
	\begin{equation}\label{2D endpoint}
		\|e^{it\Delta}f\|_{L_x^3L_t^{\infty}(P_R)}\lesssim_{\epsilon} N^{\frac{1}{3}}R^{\epsilon}||f||_{L^2}.
	\end{equation}
\end{theorem}

\begin{theorem}[High frequency input estimate for $n\geq3$]\label{t4.6}
	Let $\hat{f}$ be supported in $ \mathbb{B}^n(Ne_1,1)$ with $N\geq 1$. For any $\epsilon>0$, we have
	
	\begin{equation}\label{n-endpoint}
		\|e^{it\Delta}f\|_{L_x^pL_t^{\infty}(P_R)}\lesssim_{\epsilon} N^{\frac{1}{p}}R^{\epsilon}||f||_{L^2},
	\end{equation}
	whenever $$p\geq \tilde{p}_n=2+\frac{4}{n+1+s_2^n}  \ \text{with}\ \  s_2^n=\sum_{i=2}^{n}\frac1i.$$
\end{theorem}

 \begin{remark}The  propagation speed of the Schr\"odinger wave $e^{it\Delta}f$ depends on the location of $\hat{f}$.  In this case, the Schr\"odinger wave $e^{it\Delta}f$ roughly propagates at speed $N$ in the direction $(4\pi Ne_1,1)$. Then for $|t|\leq R$, the propagation region of the wave is essentially concentrated in $P_R$. Since there is no Galilean invariant, the estimate with different frequency input is inevitable.
 \end{remark}



\subsection{Global space-time estimates}
We  turn back  to Conjecture \ref{c1.1}. For $n=1$, it was proved by Kenig, Ponce, and Vega \cite{kpv91}. For $n\geq 2$,  Vega \cite{v88} set up \eqref{S_{p,r}}  for $(p,r)=(\frac{2(n+1)}{n-1},2)$ by Plancherel's theorem (in time) and the $n-1$ dimensional Stein-Tomas restriction theorem. Combining his result with the $n$-dimensional Stein-Tomas restriction theorem, one can verify the estimate \eqref{S_{p,r}} for $p\geq \frac{2(n+2)}{n}$ in the critical case. It remains to verify Conjecture \ref{c1.1} when $\frac{2(n+1)}{n}\leq p\leq \frac{2(n+2)}{n}$.  Combining Tao's sharp bilinear restriction estimates for the paraboloid \cite{t03}  with a bilinear-to-linear argument of  Planchon \cite{p02}, one can obtain \eqref{S_{p,r}} in the subcritical region and $p>\frac{2(n+3)}{n+1}$. This result was latter extended to the critical case by Lee, Rogers, and Vargas \cite{lrv11}. When $n=2$, there are further results that breakthrough the index $\frac{2(n+3)}{(n+1)}=\frac{10}{3}$. Rogers \cite{r09} obtained \eqref{S_{p,r}} for $p>\frac{16}{5}$ in the subcritical case.  Lee, Rogers, and Vargas \cite{lrv11} extend it to the critical case. Very recently, Cho, Koh, and Lee \cite{ckl22} set up \eqref{S_{p,r}} in the whole subcritical region.
With theorems \ref{t2.1} and \ref{t4.6} at hand, we can obtain the following space-time estimates in the critical case.

\begin{theorem}\label{t1.1}
	Let  $2\leq r<\infty$ and $\frac{n+1}{p}+\frac{1}{r}=\frac n2$. Then
\begin{itemize}
	\item[\rm{(i)}] when $n=2$, \eqref{S_{p,r}} holds for $p>3$;
	\item[\rm{(ii)}] when $n\geq 3$, \eqref{S_{p,r}} holds  for $p> p_n$, where
	$$p_n=\frac{2(n+1)[(n+2)\tilde{p}_n-2(n+3)]}{(n+1)n\tilde{p}_n-2(n^2+2n-1)},$$
with the same $\tilde{p}_n$ as in Theorem \ref{t4.6}.
\end{itemize}
\end{theorem}
\begin{remark}\label{re1.4}
	By this theorem, we confirm Conjecture \ref{c1.1} when $n=2$ (see Figure 1). For $n\geq 3$, we improve the  results obtained Lee, Rogers, and Vargas \cite{lrv11} (see the green part of Figure 2). In Table \ref{tab1}, we list several values of $p_n$, the conjectured sharp index $p_s$, and $p_{lrv}$ obtained by Lee, Rogers, and Vargas \cite{lrv11}.

\end{remark}
\begin{remark}
	The local maximal estimate \eqref{DGL} is not enough to obtain critical estimate \eqref{S_{p,r}} for $n=2$ (see Cho, Koh, and Lee \cite{ckl22}).  
	\end{remark}

\begin{figure}
	\begin{minipage}[t]{0.5\linewidth}\label{10}
		\begin{tikzpicture}
			\draw[->,thick] (0,0) -- (5.2,0) node[right] {$\frac{1}{p}$};
			\draw[->,thick] (0,0) -- (0,5.2) node[above] {$\frac{1}{r}$}; 
			\draw (0,0) node[anchor=north east] {O};
			\draw (3.33,0) circle (.1);
			\filldraw[red,line width=1] (0,0) -- (0,5) -- (1.67,5) -- (3.33,0);
			\draw (1.67,5) -- (3.33,0);
			\draw (1.675,5) - - (1.675,0);
			\draw (1.675,0) node[below] {$\frac16$};
			\draw (3.33,0) node[below] {$\frac13$};
			\draw (5,0) node[below] {$\frac 12$};
			\draw (0,5) node[left] {$\frac 12$};
		\end{tikzpicture}
		\caption{$n=2$}
	\end{minipage}%
	\begin{minipage}[t]{0.5\linewidth}\label{11}
		\begin{tikzpicture}
			\draw[->,thick] (0,0) -- (5.2,0) node[right] {$\frac{1}{p}$};
			\draw[->,thick] (0,0) -- (0,5.2) node[above] {$\frac{1}{r}$};
			\draw (0,0) node[anchor=north east] {O}; 
			\filldraw[red,line width=1] (0,0) -- (0,5) -- (2.5,5) -- (3.3,1.8)-- (3.3,0);
			\filldraw[green] (3.3,1.8) --(3.57,0.7)--(3.57,0)--(3.3,0);
			\draw (2.5,5) -- (3.75,0);
			\draw [black,thick, dotted] (3.57,0.7)-- (3.57,0);
			\draw [black,thick, dotted] (3.3,1.8)-- (3.3,0);
			\draw (3.57,0.7) circle (.1);
			\draw (2.5,5) -- (2.5,0);
			\draw (2.5,0) node[below] {$\frac{n-2}{2n}$};
			\draw (3.1,0) node[below] {$\frac{1}{p_{lrv}}$};			
			\draw (3.57,0) node[below] {$\frac{1}{p_n}$};
\draw (4,0) node[below] {$\frac{1}{p_s}$};
			\draw (5,0) node[below] {$\frac12$};
			\draw (0,5) node[left] {$\frac12$};
		\end{tikzpicture}
		\caption{$n\geq3$}
	\end{minipage}
	
\end{figure}

\begin{table}[h!]
	\caption{}\label{tab1}
	
	\begin{tabular}{r|>{}c>{}r>{}c>{}c>{}r>{}c>{}c>{}r>{}c>{}c}
		$n=$    & 2 & 3 & 4 & 5 & 6 & 7 & 8 & 9 & 10 & 11\\
		\hline
		$p_{\text{s}}=$         & 3 & $2+\frac23$ & $2+\frac12$ & $2+\frac25$ & $2+\frac 13$ & $2+\frac27$ & $2+\frac14$ & $2+\frac29$ & $2+\frac 15$ & $2+\frac{2}{11}$\\
		\hline
		$p_n\sim$      & 3 & $2+\frac{44}{50}$ & $2+\frac {7}{10}$ & $2+\frac35$ & $2+\frac12$ & $2+\frac{47}{100}$ & $2+\frac25$ & $2+\frac{38}{100}$ & $2+\frac{7}{20}$ & $2+\frac{3}{10}$\\
		
		\hline
		$p_{lrv}=$      & $3+\frac15$ & $3$ & $2+\frac 45$ & $2+\frac23$ & $2+\frac47$ & $2+\frac12$ & $2+\frac49$ & $2+\frac25$ & $2+\frac{4}{11}$ & $2+\frac13 $\\

	\end{tabular}
	
\end{table}

\subsection{Local space-time estimates}
Beside the global in time estimates as above, the local in time estimates are also interesting. The following   conjecture formulated by Lee, Rogers, and Seeger \cite{lrs13} involving the Stein's square function estimate for the Bochner-Riesz mean and also the adjoint estimate of the Fourier restriction estimate.
\begin{conjecture}\label{c1.4}
	Let $2\leq p<\infty$, $2\leq r\leq\infty$ satisfy $\frac np+\frac1r<\frac n2$ and $\frac{2n+1}{p}+\frac1r<n$. Then, 
	\begin{equation}\label{ls2}
		\|e^{it\Delta}f\|_{L_x^pL_t^{r}(\mathbb{R}^n\times[0,1])}\lesssim \|f\|_{B_{\alpha,p}^p(\mathbb{R}^n)}			
	\end{equation}
	hold for $\alpha=n(1-\frac 2p)-\frac2r$. Here $B_{\alpha,p}^p$ denotes the non-homogeneous Besov space with the norm $$\|f\|_{B_{\alpha,p}^p}=(\sum_{k\geq0}2^{k\alpha p}\|P_kf\|_{p}^p )^{1/p}.$$ 
\end{conjecture}
For $n=1$, the conjecture was proved by Lee, Rogers, and Seeger \cite{lrs13}. The  higher dimensional case is still open. Figure 3 is the index graph corresponding to Conjecture \ref{c1.4}. On the line $(0,\frac12)-(\frac{n-1}{2n},\frac12)$, Conjecture \ref{c1.4} is the corollary of the Stein's square function conjecture (see \cite{lrs12}). Which was proved by Lee, Rogers, and Seeger \cite{lrs12} to be true for $p>4$ when $n=2$.  The most recent progress in higher dimension can be found in
Gan, Oh, and Wu \cite{gow21}. Combing the knowledge of  the Fourier restriction estimate, one can obtain the estimates in the red region.  For the yellow area, Lee, Rogers, and Seeger \cite{lrs13} showed that Conjecture \ref{c1.4} can be reduced to the Fourier restriction estimate conjecture that
\begin{equation}\label{restrict estimate}
	\Big\|\int_{\mathbb{B}^n(0,1)}e^{i(\xi_1y_1+\dots+\xi_ny_n+s|\xi|^2)}f(\xi)d\xi\Big\|_{L^p(\mathbb{R}^{n+1})}\lesssim \|f\|_{L^p(\mathbb{B}(0,1))}
\end{equation}
holds for $p>\frac{2(n+1)}{n}$.  In the blue area, as far as the authors knowledge, there is no result in this part.

In this paper, we focus on the  problem for $n=2$. In this case, Wang \cite{w22} proved the Fourier restriction estimate \eqref{restrict estimate}  for $p>3+\frac{3}{13}$.  It follows that \eqref{ls2} is true for $p>3+\frac{3}{13}$ and $p\leq r\leq\infty$ ($A$ of Figure 4). Combining  the results of square function estimate obtained by Lee, Rogers and Seeger \cite{lrs12}, the red region in Figure 4 is verified for \eqref{ls2}. In \cite{lrs13}, Lee, Rogers, and Seeger set up \eqref{ls2} for $p>3+\frac15$ and $r\geq 4$ ($B$ of Figure 4) which allow us to have the yellow region. In this paper, we obtain the brown region in Figure 4 by setting up the following theorem.

\begin{figure}
	\begin{minipage}[t]{0.5\linewidth}\label{12}
		\begin{tikzpicture}
			\draw[->,thick] (0,0) -- (5.2,0) node[right] {$\frac{1}{p}$};
			\draw[->,thick] (0,0) -- (0,5.2) node[above] {$\frac{1}{r}$}; 
			\draw (0,0) node[anchor=north east] {O};
			\draw [red] (0,5) -- (2.5,5);
			\draw (0,0) -- (5,5);
				\draw (2.5,5) circle (.1);
				\draw (3.33,3.33) circle (.1);
					\draw (4,0) circle (.1);
			\draw [thick, dotted] (2.5,5) -- (3.33,3.33);
			\draw [thick, dotted] (4,0) -- (3.33,3.33);
			\draw (0,5) node[left] {$1/2$};
			\draw (2.5,5) node[right] {$(\frac{n-1}{2n},\frac12)$};
			\filldraw[yellow] (3.33,3.33) --(3.33,0)--(0,0)--(3.33,3.33);
			\filldraw[blue] (3.33,3.33) --(3.33,0)--(4,0)--(3.33,3.33);
			\filldraw[red] (3.33,3.33) --(2.5,5)--(0,5)--(0,0)--(3.33,3.33);
			\draw (4,0) node[below] {$\frac{n}{2n+1}$};
			\draw (3.33,3.33) node[right] {$(\frac{n}{2(n+1)},\frac{n}{2(n+1)})$};
		\end{tikzpicture}
		\caption{ }
	\end{minipage}%
	\begin{minipage}[t]{0.5\linewidth}\label{13}
		\begin{tikzpicture}
			\draw[->,thick] (0,0) -- (5.2,0) node[right] {$\frac{1}{p}$};
			\draw[->,thick] (0,0) -- (0,5.2) node[above] {$\frac{1}{r}$}; 
			\draw (0,0) node[anchor=north east] {O};
			\draw (0,5) -- (2.5,5);
			\draw (0,0) -- (5,5);
			\draw (0,2.5) -- (3.1,2.5);
			\draw [thick, dotted] (3.1,2.5) -- (3.1,0);
			\draw (0,1.82) -- (3.25,1.82);
			\draw [thick, dotted](2.5,5) -- (2.9,2.9);
			\draw [thick, dotted] (3.25,1.82) -- (3.25,0);
			\draw [thick, dotted] (2.9,2.9) -- (2.9,0);
			\draw [thick, dotted] (2.5,5) -- (3.33,3.33);
			\draw [thick, dotted] (4,0) -- (3.33,3.33);
			\draw (0,5) node[left] {$1/2$};
			\draw (2.5,5) node[right] {$(\frac14,\frac12)$};
			\draw (3.1,2.5) node[right] {B};
			\draw (3.25,1.82) circle (.1);
			\draw (3.1,2.5) circle (.1);
			\draw (2.9,2.9) circle (.1);
			\draw (3.33,3.33) node[right] {$(\frac13,\frac13)$};
			\draw (3.25,1.82) node[right] {C};
			\draw (2.9,2.9) node[right] {A};
			\draw (4,0) node[below] {$\frac25$};
		    \filldraw[red] (0,5) --(2.5,5)--(2.9,2.9)--(2.9,0)--(0,0);
		    \filldraw[yellow] (2.9,2.9) --(2.9,0)--(3.1,0)--(3.1,2.5);
		    \filldraw[brown](3.1,2.5)--(3.1,0)--(3.25,0)--(3.25,1.82);

		\end{tikzpicture}
		\caption{$A=(\frac{13}{42},\frac{13}{42})$,\ $B=(\frac{5}{16},\frac 14)$,\ $C=(\frac{7}{22},\frac{2}{11})$ }
	\end{minipage}%

\end{figure}

\begin{theorem}\label{c1.3} When $n=2$, the local space-time estimate \eqref{ls2} holds for $p>3+1/7$ and $r\geq 5+1/2.$
\end{theorem}

\begin{remark}
 In higher dimensional case, the same argument as in Section \ref{sub.2} allows us to verify \eqref{ls2} for some $(p,r)$. But they are still in the region  implied by the Fourier restriction estimate obtained by Hickman and Zahl \cite{hz20}.
\end{remark}

This paper is organized as follows. In Section \ref{2D case},  Theorem \ref{t2.1} is proved by using the polynomial partitioning of Guth \cite{G16} and the bilinear refined Strichartz estimate of Du, Guth and Li \cite{dgl17}. In Section \ref{3D case}, Theorem \ref{t2.2x}, a general form of Theorem \ref{t4.6},  is reduced to a  $2$-broad norm $BL_{A}^{p,q}$ estimate. For later use, a $BL_{A}^{2,\infty}$ estimate with input function tangent to an $m$-variety is also studied. In Section \ref{Alg}, we set up the $BL_{A}^{p,q}$ estimate by two algorithms introduced by Hickman and Rogers \cite{hg19}. In section \ref{sub.1}, Theorem \ref{t1.1} is proved by using Theorem \ref{t2.1}. In section \ref{sub.2}, Theorem \ref{c1.3} is proved. In Appendix A, we give a different proof of the epsilon removal for the bilinear space-time estimate from Lee, Rogers and Vargas \cite{lrv11}. In Appendix B, we formulate a conjecture on the bilinear space-time estimate which implies Conjecture \ref{c1.1}.  

 Notation list:
\begin{itemize}
	\item $A\lesssim B$: there exists a positive constant $C$ such that $A\leq CB$, where $C$ may change from line to line. $A\sim B$ means $A\lesssim B $ and $B \lesssim A$.
	\item $A\lesssim_\epsilon B$: the implicant constant $C$  depends on a given $0<\epsilon\ll1$.
	\item $A(R)\leq \text{RapDec}(R)B$: for any $N\geq 1$, there exists a constant $C_N$ such that
	$$A(R)\leq C_N R^{-N}B\ \ \text{for all}\ R\geq 1.$$
    \item For a finite set $A$, $\sharp A$  denotes its cardinality.
    \item  $\epsilon>0$ is a very small number and $\delta\ll\epsilon$ (for example $\delta=\epsilon^2$).
    \item $R\geq 1$ is an arbitrary very large number. $r$ denotes some number $1\leq r\leq R$. $K$ is also very large but far less than $R$.
    \item $\mathbb{B}^n(Ne_1,1):=\{\xi\in\mathbb{R}^n: |\xi-Ne_1|\leq 1\}$, here $N\geq 1$ and $e_1=(1,0,\dots,0)$.
    \item Let $z_0=(x_0,t_0)\in \mathbb R^{n+1}$ and $r\geq 1$, $$P_r(z_0):=\{(x,t)\in \mathbb{R}^n\times [-r+t_0,r+t_0]:x-4\pi tNe_1\in Q_r(x_0)\}.$$ Here $Q_r(x_0)\subset\mathbb R^n$ denotes a cube centered at $x_0$ with side length $r$. The parameter $r$ is the size of $P_r(z_0)$. We abbreviate it as $P_r$ if the center is not important.
    \item  $T_{\theta,\nu}\in \mathbb T[r]$ means that a tube 
    $$T_{\theta,\nu}:=\Big\{(x,t)\in\mathbb{R}^n\times\mathbb{R}\Big|\ 0\leq t\leq r,\ |x-4\pi tc(\theta)-c(\nu)|\leq {r}^{\frac{1}{2}+\delta}\Big\},$$
   here $c(\theta)\in {r}^{-\frac{1}{2}}\mathbb{Z}^{n}$ and $ c(\nu)\in {r}^{\frac{1}{2}}\mathbb{Z}^{n}$.
   \item $Z$ or $Z(P)$ denotes the variety of a polynomial $P$. $T_zZ$ denotes the tangent space of $Z$ at point $z\in Z$.
\end{itemize}




\section{High frequency input endpoint estimate for $n=2$}\label{2D case}
We devote this section to prove Theorem \ref{t2.1} by an induction argument. Moreover, we set up the following theorem.
\begin{theorem}\label{t2.2}
Let $p>3$ and $\ M\geq 1$, $f\in L^2(\mathbb{R}^2)$ with $\text{supp} \hat{f}\subset \mathbb{B}^2(\xi_0, \frac{1}{M})\subset \mathbb{B}^2(Ne_1,1)$ for some $N\geq 1$.
For any $\epsilon>0$, $R\geq 1$, and $r>1/\epsilon^4$, there exists a constant $C_{p,\epsilon}$ such
	that
\begin{equation}\label{ISTE}
	\|e^{it\Delta}f\|_{L_x^{p}L_t^{r}(P_R)}\leq C_{p,\epsilon} N^{\frac{1}{p}-\frac{1}{r}}M^{-\epsilon^2}R^{\epsilon}\|f\|_{L^2}.
\end{equation}
\end{theorem}
 Since the constant $C_{p,\epsilon}$ does not depend on $r$, by taking $r\to \infty$ and $M=1$, we have
$$\|e^{it\Delta}f\|_{L_x^{p}L_t^{\infty}(P_R)}\leq C_{p,\epsilon}N^{\frac{1}{p}}R^{\epsilon}\|f\|_{L^2}.$$
By interpolating with the trivial estimate $$\|e^{it\Delta}f\|_{L_x^{2}L_t^{\infty}(P_R)}\lesssim (NR^2)^{1/2}\|f\|_{L^2},$$ we arrive that
$$\|e^{it\Delta}f\|_{L_x^{3}L_t^{\infty}(P_R)}\lesssim_{\epsilon} N^{\frac{1}{3}}R^{\epsilon}\|f\|_{L^2}.$$

\subsection{Wave packet decomposition}

Wave packet decomposition is a fundamental tool in the proofs. In this subsection, we collect several basic concepts and properties of the wave packets without proof. We refer the readers  to \cite{t03} for the details. 

  Let $\phi$ be a Schwartz function satisfying $\text{supp}\, \hat{\phi}\subset\mathbb{B}^{n}(0,3/2)$ and $$\sum_{k\in\mathbb{Z}^{n}}\hat{\phi}(\xi-k)=1,\ \ \ \ \text{for all}\quad \xi\in{\mathbb{R}^n}.$$
Given $\lambda>0$, let 
\begin{equation}\label{eq:2.2}
	\widehat{\phi_{\theta}}(\xi):={\lambda}^{\frac{n}{2}}\hat{\phi}({\lambda}^{\frac{1}{2}}(\xi-c(\theta))),\ \widehat{\phi_{\theta,\nu}}(\xi):=e^{-2\pi ic(\nu)\cdot\xi}\hat{\phi_{\theta}}(\xi),
\end{equation}
where $c(\theta)\in {\lambda}^{-\frac{1}{2}}\mathbb{Z}^{n}$ and $ c(\nu)\in {\lambda}^{\frac{1}{2}}\mathbb{Z}^{n}$. $\theta$ denotes a ${\lambda}^{-\frac{1}{2}}$-ball in the frequency space and $\nu$ denotes a ${\lambda}^{\frac{1}{2}}$-ball in the physical space. Usually $\phi_{\theta,\nu}$ is called as a  wave packet with scale $\lambda$.
Let $\Omega:=\text{supp}\hat{f}\subset\mathbb R^n$. A wave packet decomposition of $f$ at scale $\lambda$ is 
\begin{equation}\label{function decompose}
f=\sum_{\theta,\nu}c_{\theta,\nu}\phi_{\theta,\nu}:=\sum_{\theta,\nu}f_{\theta,\nu},\ \ \theta\cap \Omega\neq\varnothing.
\end{equation}
Then  
\begin{equation}\label{semigroup decompose}
	e^{it\Delta}f=\sum_{\theta,\nu}c_{\theta,\nu}e^{it\Delta}\phi_{\theta,\nu}:=\sum_{\theta,\nu}e^{it\Delta}f_{\theta,\nu}.
\end{equation}
These decompositions  satisfy  the following properties.
\begin{itemize}
    \item The Fourier transform of $f_{\theta,\nu}$ is supported on $\theta$, and the Fourier transform of $e^{it\Delta}f_{\theta,\nu}$ is supported on the paraboloid $$\big\{(\xi,\tau);\tau=2\pi|\xi|^2, \xi\in\theta\big\}.$$
	\item The functions $\phi_{\theta,\nu}$ are approximately orthogonal, and
	\begin{equation}\label{e3.1}
		\sum_{\theta,\nu}|c_{\theta,\nu}|^{2}\approx\sum_{\theta,\nu}\|f_{\theta,\nu}\|_{L^2}^{2}\approx\|f\|^{2}_{L^{2}}.
	\end{equation}
	
	\item By the stationary phase method,
	\begin{equation}\label{e3.2}
		|e^{it\Delta}\phi_{\theta,\nu}|\leq {\lambda}^{-\frac{n}{4}}\chi_{T_{\theta,\nu}}(x,t)+\text{RapDec}(\lambda),\  0\leq t\leq \lambda.
	\end{equation}
	Here 
	$$T_{\theta,\nu}:=\Big\{(x,t)\in\mathbb{R}^n\times\mathbb{R}\Big|\ 0\leq t\leq \lambda,\ |x-4\pi tc(\theta)-c(\nu)|\leq {\lambda}^{\frac{1}{2}+\delta}\Big\},$$
	where $\delta \ll\epsilon$. $T_{\theta,\nu}$ is a tube of length $\lambda$, of
	radius ${\lambda}^{\frac{1}{2}+\delta}$ with the direction $(4\pi c(\theta),\ 1)$. The collection of such tubes is denoted as $\mathbb T[\lambda]$ which is also used to denote the set of the $(\theta,\nu)$ pairs.
\end{itemize}

Since the Fourier transform of $e^{it\Delta}f$ with $\hat{f}\subset \mathbb{B}^n(Ne_1,1)\,(N\geq 1)$ is supported on a rectangular box $1\times ...\times 1\times N$. By the locally constant property, $e^{it\Delta}f$ is roughly constant on a rectangular box $1\times ...\times 1\times 1/N$. Similarly, for each ball $\tau\subset \mathbb{B}^n(Ne_1,1)$ with radius $K^{-1}$, $e^{it\Delta}f_{\tau}$ (with $\text{supp}\hat{f}_\tau\subset\tau$) is locally constant on a rectangular box $K\times ...\times K\times K/N$.

\subsection{The induction argument and the polynomial partitioning}

Theorem \ref{t2.2} will be proved by induction on the physical radius $R$ and the frequency radius $\frac{1}{M}$.  

When $R\lesssim 1$, it's easy to see that
	$$\|e^{it\Delta}f\|_{L_x^{p}L_t^{r}(P_1)}\leq C_{p,\epsilon} N^{\frac{1}{p}-\frac{1}{r}}M^{-\epsilon^2}\|f\|_{L^2}.$$

When $R\gg 1$. If $M>R^{10}$, Theorem \ref{t2.2} is  trivial.
	If $R^{\frac{1}{2}-\delta}\leq M\leq R^{10}$, we adopt the wave packet decomposition \eqref{function decompose} at scale $R$. Since supp$\hat{f}\subset B(\xi_0,\frac{1}{M})$, we have
	$$\sharp\Big\{\theta: \exists \nu\ s.t.\ (\theta,\nu)\in \mathbb{T}\Big\}\leq \frac{(1/M)^2}{1/R}\lesssim R^{O(\delta)}.$$
	Then, for $p>3$ and $r>\epsilon^{-4}$,
	\begin{align}\label{r1.1}
		||e^{it\Delta}f||_{L_x^pL_t^r(P_R)}
		&=\Big\|\Big|\sum_{\theta,\nu}c_{\theta,\nu}e^{it\Delta}\phi_{\theta,\nu}\Big|^2\Big\|_{L_x^{\frac{p}{2}}L_t^{\frac{r}{2}}(P_R)}^{\frac{1}{2}} \nonumber\\
		&\lesssim R^{O(\delta)}\sup_{\theta_0}\Big\|\sum_{\nu}\Big|c_{\theta_0,\nu}e^{it\Delta}\phi_{\theta_0,\nu}\Big|^2\Big\|_{L_x^{\frac{p}{2}}L_t^{\frac{r}{2}}(P_R)}^{\frac{1}{2}}\nonumber\\
		& \lesssim R^{O(\delta)} \sup_{\theta_0}R^{-\frac{1}{2}}\Big(\sum_{\nu}|c_{\theta_0,\nu}|^2\|\chi_{T_{\theta_0,\nu}}\|_{L_x^{\frac{p}{2}}L_t^{\frac{r}{2}}(P_R)}\Big)^{\frac{1}{2}}\nonumber\\
		& \lesssim R^{O(\delta)}R^{-\frac{1}{2}}(N^{-1}R^{\frac{1}{2}})^{\frac{1}{r}}(R^{\frac{3}{2}}N)^{\frac{1}{p}} \sup_{\theta_0}\Big(\sum_{\nu}|c_{\theta_0,\nu}|^2\Big)^{\frac{1}{2}}\nonumber\\
		& \lesssim R^{O(\delta)}N^{\frac{1}{p}-\frac{1}{r}} R^{\frac{3}{2p}-\frac{1}{2}+\frac{1}{2r}}||f||_{L^2}\nonumber\\
		& \lesssim N^{\frac{1}{p}-\frac{1}{r}}M^{-\epsilon^2}R^{\epsilon}\|f\|_{L^2}.
	\end{align}
Therefore it is sufficient to consider the case $R\gg1$ and $M\ll R^{\frac{1}{2}}$.

\begin{description}
\item[Induction Assumption] \eqref{ISTE} holds for $1\leq \tilde{R}\leq R/2$ or $\tilde{M}\geq 2M.$
\end{description}

We now introduce the polynomial partitioning. For a polynomial $P$, it is called nonsingular if $\nabla P(z)\neq 0$ for each point $z\in Z(P)$, where $Z(P):=\{(x,t)\in\rn\times\mathbb{R}:P(x,t)=0\}$ denotes a variety of the polynomial $P$. The following polynomial partitioning theorem on mixed norm is obtained by Du, Guth, and Li.
\begin{theorem}\cite{dgl17}\label{t2.3}
 If $F\in L^1_xL^r_t(\rn\times\mathbb{R})\backslash \{0\},\ 1\leq r<\infty$, then for each $D$, there exists a nonzero polynomial $P$ of degree at most $D$ such that $(\rn\times\mathbb{R})\backslash Z(P)$ is a union of $\sim_n D^{n+1}$ disjoint open sets  $O_i$ and for each $i$, we have
$$\|F\|_{L^1_xL^r_t(\rn\times\mathbb{R})}\leq c_nD^{n+1}\|\chi_{O_i} F\|_{L^1_xL^r_t(\rn\times\mathbb{R})}.$$
Moreover, the polynomial $P$ is a product of distinct nonsingular polynomials. The open set $O_i$ is called as a cell.
\end{theorem}
Therefor, there exists a nonzero polynomial $P$ with degree $D=R^{\epsilon^4}$ such that $(\mathbb{R}^2\times\mathbb{R})\backslash \{Z(P)\}$ is a union of $\sim D^3$ disjoint cells  $O_i$ and 
\begin{equation}\label{e2.1}
	\|e^{it\Delta}f\|_{L_x^{p}L_t^{r}(P_R)}^p\lesssim D^3 \|\chi_{O_i}e^{it\Delta}f\|_{L_x^{p}L_t^{r}(P_R)}^p,\quad \forall\quad 1\leq i\lesssim D^3.
\end{equation}

A wall is the set $W:=N_{R^{1/2+\delta}}Z(P)\cap P_R$ with $\delta=\epsilon^2$, where $N_{R^{1/2+\delta}}Z(P)$ stands for the $R^{1/2+\delta}$-neighborhood of the variety $Z(P)$. For each cell $O_i$, we define
\begin{equation}
	O_i^{'}:=[O_i\cap P_R]\backslash W,\ \ \mathbb T_i:=\{(\theta,\nu)\in \mathbb T[R]:T_{\theta,\nu}\cap O_i^{'}\neq \emptyset \}\ \ \text{and}\ \ f_i:=\sum_{(\theta,\nu)\in \mathbb T_i }f_{\theta,\nu}.
\end{equation}
For each shrunken cell $O'_i$, we have
$$\chi_{Q'_i}e^{it\Delta}f\sim e^{it\Delta}f_i.$$
With these notations, 
\begin{equation}\label{eq:3.5}
	\|e^{it\Delta}f\|_{L_x^{p}L_t^{r}(P_R)}^p\lesssim \sum_i\|\chi_{O_i^{'}}e^{it\Delta}f\|_{L_x^{p}L_t^{r}(P_R)}^p+\|\chi_{W}e^{it\Delta}f\|_{L_x^{p}L_t^{r}(P_R)}^p.
\end{equation}
The estimate of \eqref{eq:3.5} will be catalogued into two cases:
\begin{description}
		\item[Cellular case] there exist $O(D^3)$ many cells $O_i^{'}$ such that for each $i$,
\begin{equation}\label{eq:2.10}
	\|e^{it\Delta}f\|_{L_x^{p}L_t^{r}(P_R)}^p\lesssim D^3\|\chi_{O_i^{'}}e^{it\Delta}f\|_{L_x^{p}L_t^{r}(P_R)}^p.
\end{equation}

\item[Wall case]
\begin{equation}\label{eq:2.11}
\|e^{it\Delta}f\|_{L_x^{p}L_t^{r}(P_R)}^p\lesssim \|\chi_{W}e^{it\Delta}f\|_{L_x^{p}L_t^{r}(P_R)}^p.
\end{equation}	
\end{description}

For the cellular case, by the geometrical fact that each $(\theta,\nu)\in \mathbb T[R]$ intersects at most $D+1$ cells $O_i^{'}$ , then 
$$\sum_i\|f_i\|_{L^2}^2\lesssim \sum_i\sum_{(\theta,\nu)\in\mathbb T_i}\|f_{\theta,\nu}\|_{L^2}^2\lesssim D\sum_{\theta,\nu}\|f_{\theta,\nu}\|_{L^2}^2\sim D\|f\|_{L^2}^2.$$ Then there exists $i_0$ satisfying \eqref{eq:2.10} and 
$$\|f_{i_0}\|_{L^2}\lesssim D^{-1}\|f\|_{L^2}.$$
We break  $P_R$ into $O(1)$ parallelepipeds $P_j$ of size $R/2$, each has the same direction with $P_R$. By the induction assumption, we have
\begin{align*}
	\|e^{it\Delta}f\|_{L_x^{p}L_t^{r}(P_R)}^p
	& \lesssim D^3\|\chi_{O_{i_0}^{'}}e^{it\Delta}f\|_{L_x^{p}L_t^{r}(P_R)}^p\lesssim D^3\|e^{it\Delta}f_{i_0}\|_{L_x^{p}L_t^{r}(P_R)}^p\\
	& \lesssim D^3\max_j\|e^{it\Delta}f_{i_0}\|_{L_x^{p}L_t^{r}(P_j)}^p\lesssim D^3\Big(C_{p,\epsilon}N^{\frac{1}{p}-\frac{1}{r}}M^{-\epsilon^2}(R/2)^{\epsilon}\|f_{i_0}\|_{L^2}\Big)^p\\
	& \lesssim D^{3-p}2^{-\epsilon p}\Big(C_{p,\epsilon}N^{\frac{1}{p}-\frac{1}{r}}M^{-\epsilon^2}R^{\epsilon}\|f\|_{L^2}\Big)^p.
\end{align*}
Since $D=R^{\epsilon^4}$ and $p>3$, with $R$ large enough, we can close the induction.

To deal with the wall case,  the set $P_R$ is broken into $R^{3\delta}$ parallelepipeds $P_{R^{1-\delta}}(z_j)$ of size $R^{1-\delta}$. Denote $P_j=P_{R^{1-\delta}}(z_j)$. For any $T_{\theta,\nu}$ with $(\theta,\nu)\in \mathbb{T}[R]$, there are two more conceptions: tangent tube and transverse tube.
\begin{description}
	\item[Tangent tube] $T_{\theta,\nu}$ is $R^{-\frac{1}{2}+\delta}$ tangent to the variety $Z(P)$ on a parallelepiped  $P_j$ if  $T_{\theta,\nu}\cap P_j\cap W\neq \emptyset$ and
$$\angle(G_0(\theta),T_z[Z(P)])\leq R^{-\frac{1}{2}+\delta}$$
for any non-singular point $z\in 10T_{\theta,\nu}\cap 2P_j\cap Z(P)$.
\item[Transverse tube] $T_{\theta,\nu}$ is $R^{-\frac{1}{2}+\delta}$ transverse to the variety $Z(P)$ on a parallelepiped $P_j$ if  $T_{\theta,\nu}\cap P_j\cap W\neq \emptyset$ and
$$\angle(G_0(\theta),T_z[Z(P)])> R^{-\frac{1}{2}+\delta}$$
for some non-singular point $z\in 10T_{\theta,\nu}\cap 2P_j\cap Z(P)$.
\end{description}
Here $G_0(\theta)=(4\pi c(\theta),1)$ is the direction of the tube $T_{\theta,\nu}$. $T_z[Z(P)]$ stands for the tangent space of $Z(P)$ at the point $z$.
Let $\mathbb{T}_{j,tang}$ represent the set of the tubes tangent to  $Z(P)$ in $P_j$, and $\mathbb{T}_{j,trans}$ denote the set of the tubes transverse to $Z(P)$ in $P_j$. Set $$f_{j,tang}:=\sum_{(\theta,\nu)\in\mathbb{T}_{j,tang}}f_{\theta,\nu},\quad f_{j,trans}:=\sum_{(\theta,\nu)\in\mathbb{T}_{j,trans}}f_{\theta,\nu}.$$   We have, for $(x,t)\in P_j\cap W$,
$$e^{it\Delta}f(x)\sim e^{it\Delta}f_{j,tang}(x)+e^{it\Delta}f_{j,trans}(x).$$

 Let parameter $K\gg1$. We decompose 
$\mathbb{B}^n(\xi_0, \frac{1}{M})$ into balls $\tau$ of radius $(KM)^{-1}$ and $$f=\sum_{\tau}f_{\tau},\,\quad\text{supp}\hat{f_\tau}\subset \tau.$$ We introduce in the following functions
\begin{equation}
	f_{\tau,j,tang}=\sum_{\theta\in\tau}\sum_{(\theta,v)\in\mathbb{T}_{j,trans}}f_{\theta,v},\quad f_{\tau,j,trans}=\sum_{\theta\in\tau}\sum_{(\theta,v)\in\mathbb{T}_{j,tang}} f_{\theta,v}.
\end{equation}
Now we consider two cases.
\begin{itemize}
	\item If
\begin{equation}\label{e2.11}
	K^{\epsilon^4}\max_{\tau}|e^{it\Delta}f_{\tau}(x)|>|e^{it\Delta}f(x)|\ \ \text{for}\ \ (x,t)\in W.
\end{equation}
\end{itemize}

By the induction assumption,  
\begin{align*}
	\|\max\limits_{\tau}|e^{it\Delta}f_{\tau}|\|_{L_x^pL_t^r(P_R)}
	& \leq\Big\|\Big(\sum_{\tau}|e^{it\Delta}f_{\tau}|^p\Big)^{\frac{1}{p}}\Big\|_{L_x^pL_t^r(P_R)}\\
	& \leq\Big(\sum_{\tau}||e^{it\Delta}f_{\tau}||_{L_x^pL_t^r(P_R)}^p\Big)^{\frac{1}{p}}\\
	& \lesssim \Big(\sum_{\tau}(C_{p,\epsilon}N^{\frac1p-\frac1r}(KM)^{-\epsilon^2}R^{\epsilon}\|f_{ \tau}\|_{L^2})^p\Big)^{\frac{1}{p}}\nonumber\\
	& \lesssim C_{p,\epsilon} K^{-\epsilon^2}N^{\frac1p-\frac1r}M^{-\epsilon^2}R^{\epsilon}\|f\|_{L^2}.
\end{align*}
From this estimate and \eqref{e2.11}, it follows that 
$$\|\chi_{W}e^{it\Delta}f\|_{L_x^{p}L_t^{r}(P_R)}^p\lesssim C_{p,\epsilon}K^{p(\epsilon^4-\epsilon^2)}(N^{\frac1p-\frac1r}M^{-\epsilon^2}R^{\epsilon}\|f\|_{L^2})^p.$$
By taking $K$ sufficiently large, we can close the induction.

\begin{itemize}
	\item If $$
	K^{\epsilon^4}\max_{\tau}|e^{it\Delta}f_{\tau}(x)|\leq|e^{it\Delta}f(x)| \ \ \text{for}\ \ (x,t)\in W.
$$
\end{itemize}
By the broad-narrow catalogue as in \cite[Section 6]{dgl17}, we have
\begin{align}\label{eq:2.12}
	&\|\chi_{W}e^{it\Delta}f\|_{L_x^{p}L_t^{r}(P_R)}^p \\\nonumber
	\lesssim &\sum_j\|\max_{I}\chi_{P_j\cap W}e^{it\Delta}f_{I,j,trans}\|_{L_x^{p}L_t^{r}(P_j)}^p+K^{10p}\sum_{j}\|\chi_{P_j\cap W}\text{Bil}(e^{it\Delta}f_{j,tang})\|_{L_x^{p}L_t^{r}(P_j)}^p \\\nonumber
	=:& \text{Linear transverse term}+\text{Bilinear tangent term}.
\end{align}
Here $I$ denotes a subset of the collection of  $1/KM$ balls $\tau$, $$f_{I,j,trans}(x):=\sum_{\tau\in I}f_{\tau,j,trans}$$ and
$$\text{Bil}(e^{it\Delta}f_{j,tang}):=\max_{\{(\tau_1,\tau_2):dist(\tau_1,\tau_2)\geq\frac{1}{KM}\}}|e^{it\Delta}f_{\tau_1,j,tang}e^{it\Delta}f_{\tau_2,j,tang}|^{\frac{1}{2}}.$$

For the linear transverse term, Guth \cite  [Lemma 3.5]{G16} proved that for a given $(\theta,\nu)\in \mathbb T$
$$\sharp\Big\{j: (\theta,\nu)\in\mathbb{T}_{j,trans}\Big\}\leq R^{O(\epsilon^4)}.$$ 
By the induction assumption,
\begin{align*}
	\sum_j\|\max_{I}\chi_{P_j\cap W}e^{it\Delta}f_{I,j,trans}\|_{L_x^{p}L_t^{r}(P_j)}^p 
	&\lesssim \sum_j\sum_{I\subset \mathcal{T}}\|\chi_{P_j\cap W}e^{it\Delta}f_{I,j,trans}\|_{L_x^{p}L_t^{r}(P_j)}^p\\
	&\lesssim \sum_j2^{K^2}\Big(C_{\epsilon}N^{\frac1p-\frac1r}M^{-\epsilon^2}R^{(1-\delta)\epsilon}\|f_{j,trans}\|_{L^2}\Big)^p\\
	&\lesssim 2^{K^2}R^{O(\epsilon^4-\delta\epsilon p)} \Big(C_{p,\epsilon}N^{\frac1p-\frac1r}M^{-\epsilon^2}R^{\epsilon}\|f\|_{L^2}\Big)^p.
\end{align*}
Here we used the fact $\sharp\mathcal{T}\leq 2^{K^2}$ since $\mathcal{T}$ is the possible  collection of  $I$ . We can close the induction since $\delta=\epsilon^2$ and $K\ll R$.

The estimate of the bilinear tangent term is complicated. We will set up the following estimate in the next subsection. 
\begin{proposition}[Bilinear tangent estimate]\label{p2.3}
Let $N,\, M\geq 1$. Suppose that $\xi_0\in \mathbb{B}^2(Ne_1,1)$ and that the Fourier supports of $f_i\,(i=1,2)$ are in $\mathbb{B}^2(\xi_0, \frac{1}{M})$ with distance at least $1/KM$ with the same $K=K(\epsilon)$ as above. Further more assume that each $f_i$ is concentrated in $\mathbb{T}_Z(R)$, where
$$\mathbb{T}_Z(R):=\{(\theta,\nu)\in\mathbb{T}[R]:T_{\theta,\nu}\ \text{is}\ R^{-\frac{1}{2}+\delta}\ \text{-tangent to}\ Z(P)\}.$$ Then
\begin{equation}\label{e2.8}
\||e^{it\Delta}f_1e^{it\Delta}f_2|^{\frac{1}{2}}\|_{L_x^{3}L_t^{\infty}(P_R)}\lesssim_{K,\epsilon}  N^{\frac{1}{3}} R^{\epsilon}\|f_1\|_{L^2}^\frac12\|f_2\|_{L^2}^{\frac12}.
\end{equation}	
\end{proposition}

In this proposition, we say a tube $T_{\theta,\nu}$ is $R^{-\frac{1}{2}+\delta}$ tangent to $Z(P)$ if 
$$T_{\theta,\nu}\subset N_{R^{\frac{1}{2}+\delta}}Z(P)\cap P_R$$
and 
$$\angle(G_0(\theta),T_z[Z(P)])\leq R^{-\frac{1}{2}+\delta}$$
for any non-singular point $z\in N_{2R^{\frac{1}{2}+\delta}}(T_{\theta,\nu})\cap 2P_R\cap Z(P)$.
A function $f$ is concentrated in $\boldsymbol{\Pi}\subset\mathbb{T}[R]$ if
\begin{equation}
	\sum_{(\theta,\nu)\notin \boldsymbol{\Pi}}\|f_{\theta,\nu}\|_{L^2}\lesssim \text{RapDec}(R)\|f\|_{L^2}.
\end{equation}

In order to apply Proposition \ref{p2.3} to deal with $\text{Bil}(e^{it\Delta}f_{j,tang})$, wave packets decomposition at different scales will be considered.

Let $R^{\frac12}\leq r\leq R$ and $P_r(z_j)$ with $z_j=(x_j,t_j)\in P_R$. Define
$$e^{it\Delta}f(x)=e^{i\tilde{t}\Delta}\tilde{f}(\tilde{x}),$$
where $(x,t)=(\tilde{x},\tilde{t})+(x_j,t_j)$, $(\tilde{x},\tilde{t})\in P_r$ and 
$$\hat{\tilde{f}}(\xi)=e^{2\pi i (x_j\cdot\xi-2\pi t_j|\xi|^{2})}\hat{f}(\xi).$$
 $\tilde{f}$ has the following wave packets decomposition
$$\tilde{f}=\sum_{\tilde{\theta},\tilde{\nu}\in\mathbb{T}[r]}\tilde{f}_{\tilde{\theta},\tilde{\nu}}+\text{RapDec}(R)\|f\|_{L^2}.$$
This fact was proved by  Guth.
\begin{lemma}\cite{G18}\label{l2.5}
	If $f$ is concentrated in $\boldsymbol{\Pi}\subset\mathbb{T}(R)$, then $\tilde{f}$ is concentrated in $\tilde{\boldsymbol{\Pi}}\subset\mathbb{T}(r)$ with the following property. For every $(\tilde{\theta},\tilde{\nu})\in \tilde{\boldsymbol{\Pi}}$, there exist a $(\theta,\nu)\in \boldsymbol{\Pi}$ such that
	$$\text{dist}_H(T_{\tilde{\theta},\tilde{\nu}}+z_j, T_{\theta,\nu}\cap P_r(z_j))\lesssim R^{\frac12+\delta}\ \ \text{and}\ \ \angle(G_0(\theta),G_0(\tilde{\theta}))\lesssim r^{-\frac12},$$
	where $\text{dist}_H$ denotes the Hausdorff distance.
\end{lemma}

Let $g=f_{j,tang}$ and $g_i=f_{\tau_i,j,tang}\ (i=1,2)$. From the definition of $(\theta,\nu)\in\mathbb{T}_{j,tang}$ and Lemma \ref{l2.5}, we obtain $\tilde{g}$ is concentrated in $\mathbb{T}_Z(R^{1-\delta})$ in $P_j$. By Proposition \ref{p2.3}, we have
\begin{align*}
\|\chi_{P_j\cap W}\text{Bil}(e^{it\Delta}f_{j,tang})\|_{L_x^{3}L_t^{\infty}(P_j)} 
	&\lesssim \||e^{it\Delta}\tilde{g}_1e^{it\Delta}\tilde{g}_2|^{\frac{1}{2}}\|_{L_x^{3}L_t^{\infty}(P_{R^{1-\delta}})} \\
	&\lesssim_{K,\epsilon} N^{\frac{1}{3}}R^{\epsilon}\|f_{\tau_1,j,tang}\|_{L^2}^\frac12\|f_{\tau_2,j,tang}\|_{L^2}^\frac12.
\end{align*}
Interpolating the bilinear estimate with $$\|\text{Bil}(e^{it\Delta}f_{j,tang})\|_{L_{x,t}^{\infty}(P_j)}\lesssim_{\epsilon} M^{-\frac12}\|f_{\tau_1,j,tang}\|_{L^2}^\frac12\|f_{\tau_2,j,tang}\|_{L^2}^\frac12,$$ we have
\begin{equation}
	\|\chi_{P_j\cap W}\text{Bil}(e^{it\Delta}f_{j,tang})\|_{L_x^{p}L_t^{\infty}(P_R)}\lesssim_{p,\epsilon} N^{\frac{1}{p}}M^{-\frac12(1-\frac{3}{p})}R^{\frac{\epsilon}{2}}\|f_{\tau_1,j,tang}\|_{L^2}^\frac12\|f_{\tau_2,j,tang}\|_{L^2}^\frac12
\end{equation}
 for $p>3$.
By H$\ddot{\text{o}}$lder's inequality, for any $r>1/\epsilon^4$, we have
\begin{align*}
	\|\chi_{P_j\cap W}\text{Bil}(e^{it\Delta}f_{j,tang})\|_{L_x^{p}L_t^{r}(P_j)}
	& \leq (R/N)^{\frac1r}\|\chi_{P_j\cap W}\text{Bil}(e^{it\Delta}f_{j,tang})\|_{L_x^{p}L_t^{\infty}(P_R)}\\
	& \lesssim_{p,\epsilon} N^{\frac{1}{p}-\frac{1}{r}}M^{-\epsilon^2}R^{\frac{\epsilon}{2}}\|f_{\tau_1,j,tang}\|_{L^2}^\frac12\|f_{\tau_2,j,tang}\|_{L^2}^\frac12.
\end{align*}
Then the bilinear term in \eqref{eq:2.12} can be controlled  by
$$K^{10p}\sum_{j}\|\chi_{P_j\cap W}\text{Bil}(e^{it\Delta}f_{j,tang})\|_{L_x^{p}L_t^{r}(P_j)}^p\lesssim_{p,\epsilon} R^{O(\delta)+\frac{\epsilon p}{2}}\Big(N^{\frac{1}{p}-\frac{1}{r}}M^{-\epsilon^2}\|f\|_{L^2}\Big)^p.$$
It is enough to close the induction.

\subsection{Bilinear tangent estimate}\label{bilinear tangent estimate}
We now give the proof of Proposition \ref{p2.3}. The key point is to convert the  $L^p_xL_t^{\infty}$ estimate to $L_{x,t}^p$ estimate through the dyadic pigeonhole principle and the locally constant property. 

 For simplicity, assume $\|f_i\|_{L^2}=1,\ i=1,2$. Given a dyadic number $H$, set $$ A_H=\Big\{x:(x,t)\in P_R,\sup_{0<t\leq R}|e^{it\Delta}f_1(x)e^{it\Delta}f_2(x)|^{\frac{1}{2}}\sim H\Big\}.$$ It is easy to see that $H\lesssim 1$. The case of $H\leq R^{-100}$ is trivial. It is sufficient to consider the case $R^{-100}\leq H\leq 1$. By the pigeonhole principle, there exists a dyadic number $H$ such that
\begin{equation}\label{e3.17}
	\||e^{it\Delta}f_1e^{it\Delta}f_2|^{\frac{1}{2}}\|_{L_x^{3}L_t^{\infty}(P_R)}\lesssim (\log R)H|A_H|^{\frac13}.
\end{equation}
Since $\supp\hat{f_i}\subset \mathbb{B}^2(\xi_0, \frac{1}{M})\subset \mathbb{B}^2(Ne_1,1)\,(N\gg1)$, the Fourier transform of $e^{it\Delta}f_i$ is supported on a rectangular box of dimensions $$\frac{1}{M}\times\frac{1}{M}\times\frac{N}{M}.$$
By the locally constant property, $|e^{it\Delta}f_1(x)e^{it\Delta}f_2(x)|^{\frac{1}{2}}$ is a essentially constant on rectangular box of dimensions $M\times M\times M/N$ (A rigorous discussion can be found in \cite[Pages 635-636]{dgl17}). We can choose  $X\subset P_R$ the union of the rectangular boxes of dimensions $M\times M\times M/N$ satisfying:	

\begin{itemize}
	\item Each vertical thin tube of dimension $M\times M\times R$ contains at most one rectangular box $M\times M\times M/N$ in $X$;
	\item The projection of $X$ on $\mathbb R^2$ covers $A_H$.
\end{itemize}
By the construction of $X$, 
\begin{equation}\label{e3.18}
	|X|\sim \frac MN |A_H|.
\end{equation}

Decompose $P_R$ into parallelepipeds $P_j$ with size $R^{\frac{1}{2}}$. For any $P_j\cap X\neq \emptyset$, making use of the Galilean transform $\mathcal{G}:(x,t)\rightarrow (\tilde{x},\tilde{t})$:
$$\tilde{t}=t\ \ \text{and}\ \ \ \tilde{x}=x-4\pi tNe_1,$$ 
we obtain 
\begin{equation}\label{eq:2.15}
	\||e^{it\Delta}f_1(x)e^{it\Delta}f_2(x)|^{\frac{1}{2}}\|_{L^6(P_j)}=\||e^{i\tilde{t}\Delta}g_1(\tilde{x})e^{i\tilde{t}\Delta}g_2(\tilde{x})|^{\frac{1}{2}}\|_{L^6(Q_j)}.
\end{equation}
Here $\supp\hat{g_i}\subset \mathbb{B}^2(\tilde{\xi}_0, \frac{1}{M}),\, \tilde{\xi}_0\in \mathbb{B}^2(0,1)$, dist$(\text{supp}\hat{g}_1,\text{supp}\hat{g}_2)\gtrsim \frac{1}{KM}$, and $\|f_i\|_{L^2}=\|g_i\|_{L^2}=1$. The image $\mathcal{G}(P_j)=Q_j$ has side length $R^{\frac12}$ (see Figure 5).
The case $$\||e^{it\Delta}f_1 e^{it\Delta}f_2|^{\frac{1}{2}}\|_{L^6(P_j)}\leq R^{-C}$$ can be ignored for some large constant $C$. 
By the pigeonhole principle,  we can catalogue the $P_j$ according to
\begin{equation}\label{eq:2.16}
	\||e^{it\Delta}f_1e^{it\Delta}f_2|^{\frac{1}{2}}\|_{L^6(P_j)}\sim \gamma,
\end{equation}
with a dyadic parameter $R^{-C}<\gamma<R^C$.
Let $Y=\bigcup_{j=1}^{J}P_j$, we have
\begin{equation}\label{e3.19}
	|X|\lesssim (\log R)|X\cap Y|\lesssim (\log R)JMR.
\end{equation}
By \eqref{e3.17}, \eqref{e3.18}, \eqref{e3.19} and the locally constant property, we have
 \begin{align}\label{e2.20}
 	||e^{it\Delta}f_1 e^{it\Delta}f_2||_{L_x^{3}L_t^{\infty}(P_R)}
 	& \lesssim R^{\epsilon}H|A_H|^{\frac13}\lesssim R^{\epsilon}N^{\frac13}(M^{-1}JR)^{\frac16}H|X|^{\frac16}\nonumber\\
 	& \lesssim R^{\epsilon}N^{\frac13}(M^{-1}JR)^{\frac16}\||e^{it\Delta}f_1e^{it\Delta}f_2|^{\frac{1}{2}}\|_{L^6(Y)}\nonumber\\
 	& \lesssim R^{\epsilon}N^{\frac13}(M^{-1}JR)^{\frac16}\||e^{i\tilde{t}\Delta}g_1(\tilde{x})e^{i\tilde{t}\Delta}g_2(\tilde{x})|^{\frac{1}{2}}\|_{L^6(\tilde{Y})},
 \end{align}
where $\tilde{Y}=\bigcup_{j=1}^{J}Q_j\subset Q_R$. 

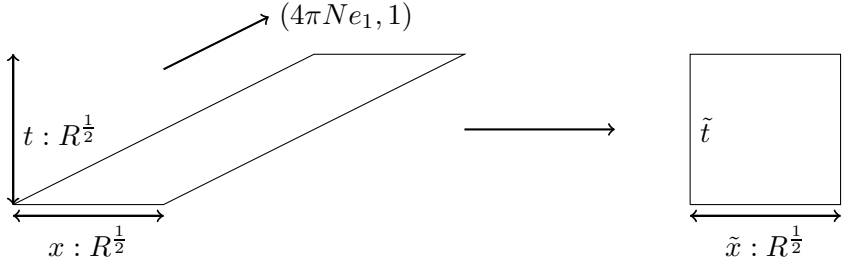
\begin{figure}
	
	\begin{tikzpicture}
		\draw  	(0,0) - - (4,2) - - (6,2) - - (2,0) - - (0,0);
		\draw[->,thick] (6,1) -- (8,1);
		\draw  	(9,0) - - (9,2) - - (11,2) - - (11,0) - - (9,0);
		\draw[< -,thick] (0,0) -- (0,2);
		\draw[ - >,thick] (0,0) -- (0,2);
		\draw[fill] (0,1) node[right] {$t:R^{\frac12}$};
		\draw[< -,thick] (0,-0.15) -- (2,-0.15);
		\draw[ - >,thick] (0,-0.15) -- (2,-0.15);
		\draw[fill] (1,-0.15) node[below] {$x:R^{\frac12}$};
		\draw[ - >,thick] (2,1.8) -- (3.4,2.5);
		\draw[fill] (3.4,2.5) node[right] {$(4\pi Ne_1,1)$};
		\draw[< -,thick] (9,-0.15) -- (11,-0.15);
		\draw[ - >,thick] (9,-0.15) -- (11,-0.15);
		\draw[fill] (10,-0.15) node[below] {$\tilde{x}:R^{\frac12}$};
		\draw[fill] (9,1) node[right] {$\tilde{t}$};
	\end{tikzpicture}
	\caption{$P_j\rightarrow Q_j$}

\end{figure}

We now recall a proposition proved by Du, Guth, and Li \cite[Proposition 8.2]{dgl17}.

\begin{proposition}\cite{dgl17}\label{p2.3x}
Let $M\geq 1$. Suppose the Fourier supports of $g_i\,(i=1,2)$ are in
$\mathbb{B}^2(\xi_0, \frac{1}{M})\subset\mathbb{B}^2(0,1)$ with the distance at least $1/KM$ for some $K=K(\epsilon)$ large enough.  $g_i$ is concentrated in $\mathbb{T}_Z(R)$. Suppose that $Q_1,Q_2,...,Q_J$ are lattice $R^{1/2}$-cubes in $Q_R$ so that
$$\||e^{it\Delta}g_1e^{it\Delta}g_2|^{\frac{1}{2}}\|_{L^6(Q_j)}\ \text{is essentially constant in}\ j.$$
Let $\tilde{Y}:=\bigcup_{j=1}^{J}Q_j$. Then 
\begin{equation}\label{eq:2.18}
	\||e^{it\Delta}g_1e^{it\Delta}g_2|^{\frac{1}{2}}\|_{L^6(\tilde{Y})}\lesssim_{K,\epsilon} M^{\frac{1}{6}}(JR)^{-\frac{1}{6}}R^{\epsilon}\|g_1\|_{L^2}^\frac12\|g_2\|_{L^2}^\frac12.
\end{equation}
\end{proposition}

The constant $K$ in Propositions \ref{p2.3x} arises in a linear-bilinear argument. We assume it is the same constant as in Proposition \ref{t2.3}. It is easy to verify that the tangent-to-variety condition is preserved under the Galilean transform. We use \eqref{eq:2.18} to bound the last term in \eqref{e2.20}, and finish the proof of Proposition \ref{p2.3}.




\section{The critical problem for $n\geq 3$ }\label{3D case}
As  Theorem \ref{t2.2}, we are going to prove the following theorem.
\begin{theorem}\label{t2.2x}
	Let $p\geq\tilde{p}_n$ and $f\in L^2(\mathbb{R}^n)$ with $\text{supp} \hat{f}\subset \mathbb{B}^n(\xi_0, \frac{1}{M})\subset \mathbb{B}^n(Ne_1,1)$ for some $N\geq 1$ and $M\geq 1$. For any $\epsilon>0$ and $R\geq 1$, there exists a constant $C_{p,\epsilon}$ such that
	\begin{equation}\label{ISTE}
		||e^{it\Delta}f||_{L_x^pL_t^{\infty}(P_R)}\leq C_{p,\epsilon} N^{\frac{1}{p}}M^{-\epsilon^2}R^{\epsilon}||f||_{L^2}.
	\end{equation}

\end{theorem}

In this section, we will introduce a $2$-broad norm $BL_A^{p,\infty}$ and collect several properties of it. Then we reduce Theorem \ref{t2.2x} to Theorem \ref{t2.3x}, whose proof will be postponed to Section \ref{Alg}. Another aim is  an $L^2_xL^\infty_t$ estimate which is also an ingredient of the induction argument.

\subsection{ \textbf{$2$-broad norm} $BL_{A}^{p,\infty}$ }

Let parameter $K\gg1$. We decompose 
$\mathbb{B}^n(\xi_0, \frac{1}{M})$ into balls $\tau$ of radius $(KM)^{-1}$ and $$f=\sum_{\tau}f_{\tau},\,\quad\text{supp}\hat{f_\tau}\subset \tau.$$
Let $$G(\tau):=\{G(\xi)\,|\,\xi\in\tau\}\quad \text{with}\quad  G(\xi)=\frac{(4\pi\xi,1)}{|(4\pi\xi,1)|}.$$
Given a linear subspace $V\subset\mathbb{R}^{n+1}$, we denote $\tau\in V$ if and only if $$\angle(G(\tau),V)<(KM)^{-1},$$ where
$$\angle(G(\tau),V)=\min\{\angle(v,v^{\prime}):\ \text{non-zero vectors}\  v,v^{\prime}\ \text{satisfy}\  v\in V\ \text{and}\ v^{\prime}\in G(\tau)\}.$$
 Otherwise, we write $\tau\notin V$. The cube $Q_R$ is decomposed into sub-cubes $Q_K$ of side length $K$. And the time interval $[0,R]$ is also decomposed into intervals $\ I_{K/N}$ of length $K/N$. Denote $$D_K:=\{x:(x,t)\in\mathcal{G}^{-1}(B_K\times I_{K/N})\},$$ then we obtain a finite overlap decomposition
$$P_R=\bigcup D_K\times I_{K/N}.$$ Given a positive integer $A$, we define
$$\mu_{e^{it\Delta}f}(D_K\times I_{K/N}):=\min_{V_1,...V_A}\Big(\max_{\tau\notin V_a\ \text{for all}\ a} \int_{D_K\times I_{K/N}}|e^{it\Delta}f_{\tau}|^pdxdt\Big),$$
where $V_1,...,V_A$ are \textbf{lines} of $\mathbb{R}^{n+1}$.  For any subset $U\subset P_R$, the $2$-broad norm $BL_{A}^{p,\infty}(U)$ is defined by
$$\|e^{it\Delta}f\|_{BL_A^{p,\infty}(U)}^p:=\sum_{D_K}\max_{I_{K/N}}\frac{|U\cap(D_K\times I_{K/N})|}{|D_K\times I_{K/N}|}\mu_{e^{it\Delta}f}(D_K\times I_{K/N}).$$

The $BL_A^{p,\infty}$-norm is not continuous, thus it can not be used to run the polynomial partitioning argument. For this purpose, we recall the $BL_A^{p,q}$-norm defined as
$$\|e^{it\Delta}f\|_{BL_A^{p,q}(U)}^p:=\sum_{D_K}\Big[\sum_{I_{K/N}}\Big(\frac{|U\bigcap(D_K\times I_{K/N})|}{|D_K\times I_{K/N}|}\mu_{e^{it\Delta}f}(D_K\times I_{K/N})\Big)^{\frac qp}\Big]^{\frac pq}.$$

The $BL_A^{p,q}$ norm has the following properties.
\begin{proposition}\cite{dl16,G18}\label{p3.2}
Let $1\leq p\leq q<\infty$. The following properties hold:
\begin{itemize}
	\item[\rm{(i)}] Let $U_1,U_2\subset\mathbb{R}^{n+1}$ and $A\in\mathbb{N}$. Then
	$$\|e^{it\Delta}f\|_{BL_A^{p,q}(U_1\cup U_2)}^p\leq \|e^{it\Delta}f\|_{BL_A^{p,q}(U_1)}^p+\|e^{it\Delta}f\|_{BL_A^{p,q}( U_2)}^p.$$
	\item[\rm{(ii)}] Let $U\subset\mathbb{R}^{n+1}$ and $A\in\mathbb{N}$. Then
	$$\|e^{it\Delta}(f_1+f_2)\|_{BL_{2A}^{p,q}(U)}\leq \|e^{it\Delta}f_1\|_{BL_{A}^{p,q}(U)}+\|e^{it\Delta}f_2\|_{BL_{A}^{p,q}(U)}$$
	\item[\rm{(iii)}] Let $U\subset\mathbb{R}^{n+1}$, $1\leq p,p_0,p_1<\infty$ and $A\in\mathbb{N}$. Suppose that $0\leq \alpha\leq 1$ satisfies
	$$\frac 1p=\frac{1-\alpha}{p_0}+\frac{\alpha}{p_1}.$$
	Then
	$$\|e^{it\Delta}f\|_{BL_{2A}^{p,q}(U)}\lesssim\|e^{it\Delta}f\|_{BL_{A}^{p_0,q}(U)}^{1-\alpha}\|e^{it\Delta}f\|_{BL_{A}^{p_1,q}(U)}^{\alpha}.$$
\end{itemize}

\end{proposition}

Moreover, $BL_A^{p,q}$ estimate can be reduced to $L_x^pL_t^{\infty}$ estimate. 
\begin{lemma}\label{l3.1}
	Let $p\geq 2$, $\epsilon>0$, and $f\in L^2(\mathbb{R}^n)$ with $\text{supp} \hat{f}\subset \mathbb{B}^n(\xi_0, \frac{1}{M})\subset \mathbb{B}^n(Ne_1,1)$. If
	$$||e^{it\Delta}f||_{L_x^pL_t^{\infty}(P_R)}\lesssim_{p,\epsilon} N^{\frac{1}{p}}W||f||_{L^2}$$
	holds for some $W=W(M,R)>0$, then
	$$\|e^{it\Delta}f\|_{BL_A^{p,\infty}(P_R)}\lesssim_{p,\epsilon} K^{O(1)}W\|f\|_{L^2}$$
	and 
	$$\|e^{it\Delta}f\|_{BL_A^{p,q}(P_R)}\lesssim_{p,\epsilon} K^{O(1)}W\|f\|_{L^2}$$
	hold for any $q>1/\epsilon^4$.
\end{lemma}
\begin{proof}
	By the definition of $BL_A^{p,\infty}$, we have 
	\begin{align*}
		\|e^{it\Delta}f\|_{BL_A^{p,\infty}(P_R)}^p
		&=\sum_{D_K}\max_{I_{K/N}}\mu_{e^{it\Delta}f}(D_K\times I_{K/N})\\
		&\leq \sum_{D_K}\max_{I_{K/N}}\sum_{\tau}\|e^{it\Delta}f_{\tau}\|_{L^p(D_K\times I_{K/N})}^p\\
		& \leq N^{-1}K^{O(1)}\sum_{\tau}||e^{it\Delta}f_{\tau}||_{L_x^pL_t^{\infty}(P_R)}^p\\
		&\lesssim_{p,\epsilon} N^{-1}K^{O(1)}\sum_{\tau}\Big(N^{\frac{1}{p}}M^{-\epsilon^2}R^{\epsilon}||f_{\tau}||_{L^2}\Big)^p\\
		&\lesssim_{p,\epsilon} K^{O(1)}\Big(W||f||_{L^2}\Big)^p.
	\end{align*}
	Since $q>1/\epsilon^4$ and the number of $I_{K/N}$ is at most $R/K$, we have $$\|e^{it\Delta}f\|_{BL_A^{p,q}(P_R)}\lesssim_{p,\epsilon} K^{O(1)}W\|f\|_{L^2}.$$ 
\end{proof}
	
\subsection{A reduction argument}
In this subsection, we will show that Theorem \ref{t2.2x} can be reduced to the following theorem.
\begin{theorem}\label{t2.3x}
Let $p\geq \tilde{p}_n$ and $f\in L^2(\mathbb{R}^n)$ with $\text{supp} \hat{f}\subset \mathbb{B}^n(\xi_0, \frac{1}{M})\subset \mathbb{B}^n(Ne_1,1)$ for some $N\geq 1$ and $M\geq 1$. For any $\epsilon>0$ and $R\geq 1$, we have	$$\|e^{it\Delta}f\|_{BL_A^{p,\infty}(P_R)}\lesssim_{p,\epsilon} M^{-\epsilon^2}R^{\epsilon}\|f\|_{L^2}.$$

\end{theorem}

\begin{proof}[Theorem \ref{t2.3x} $\Rightarrow$ Theorem \ref{t2.2x}]

When $M\geq R^{\frac{1}{2}-\delta}$, we take the wave packet decomposition at scale $R$. Since supp$\hat{f}\subset B(\xi_0,\frac{1}{M})$, we have
$$\sharp\Big\{\theta: \exists \nu\ s.t.\ (\theta,\nu)\in \mathbb{T}\Big\}\leq \frac{(1/M)^n}{(1/R)^{n/2}}\lesssim R^{O(\delta)}.$$
Then, for $p>\frac{2(n+1)}{n}$, by the same argument as \eqref{r1.1}, we have 
$$||e^{it\Delta}f||_{L_x^pL_t^{\infty}(P_R)}\lesssim N^{\frac{1}{p}}M^{-\epsilon^2}R^{\epsilon}\|f\|_{L^2}.$$

When $M< R^{\frac{1}{2}-\delta}$. We decompose
$$||e^{it\Delta}f||_{L_x^pL_t^{\infty}(P_R)}^p=\sum_{D_K}\int_{D_K}\max_{I_{K/N}}\sup_{t\in I_{K/N}}|e^{it\Delta}f(x)|^p dx.$$
For each $D_K\times I_{K/N}$, there exist lines $V_1,...,V_A$ which achieve the minimum of $\mu_{e^{it\Delta}f}(D_K\times I_{K/N})$. Then, for $(x,t)\in D_K\times I_{K/N}$, we have
$$|e^{it\Delta}f(x)|^p\lesssim K^{O(1)}\max_{\tau\notin V_a\ \text{for all}\ a}|e^{it\Delta}f(x)|^p+\sum_{a=1}^A\Big|\sum_{\tau\in V_a}e^{it\Delta}f(x)\Big|^p$$
and
\begin{align*}
	||e^{it\Delta}f||_{L_x^pL_t^{\infty}(P_R)}^p
	&\lesssim K^{O(1)}\sum_{D_K}\int_{D_K}\max_{I_{K/N}}\sup_{t\in I_{K/N}}\max_{\tau\notin V_a\ \text{for all}\ a}|e^{it\Delta}f_{\tau}(x)|^p dx\\
	&\ \ \ \ \  +\sum_{a=1}^A \sum_{D_K}\int_{D_K}\sup_{t\in [0,R]}|\sum_{\tau\in V_a}e^{it\Delta}f_{\tau}(x)|^p dx\\
	& =:I+II.
\end{align*}

For $I$, from the local constant property and Theorem \ref{t2.3x}, it follows that
\begin{align*}
	I
	&\lesssim NK^{O(1)}\sum_{D_K} \max_{I_{K/N}}\max_{\tau\notin V_a\ \text{for all}\ a}\int_{D_K\times I_{K/N}}|e^{it\Delta}f_{\tau}(x)|^p dx\\
	&\lesssim NK^{O(1)}\|e^{it\Delta}f\|_{BL_A^{p,\infty}(P_R)}^p\\
	& \lesssim K^{O(1)} \Big(N^{\frac1p}M^{-\epsilon^2}R^{\epsilon}\|f\|_{L^2}\Big)^p.
\end{align*}

For $II$, notice that $V_a\subset \mathbb{R}^n$ is a line, we have
$$\sharp\{\tau: \tau\in V_a\}\lesssim 1.$$
By induction on the frequency radius, 
\begin{align*}
	II
	&\lesssim \sum_{a=1}^A\sum_{\tau} ||e^{it\Delta}f_{\tau}||_{L_x^pL_t^{\infty}(P_R)}^p\\
	&\lesssim \sum_{a=1}^A\sum_{\tau}[N^{\frac1p}(KM)^{-\epsilon^2}R^{\epsilon}\|f\|_{L^2}]^p\\
	& \lesssim AK^{-p\epsilon^2} \Big(N^{\frac1p}M^{-\epsilon^2}R^{\epsilon}\|f\|_{L^2}\Big)^p.
\end{align*} 
We can close the induction by choosing $K$ large enough and finish the proof.	
\end{proof}

\subsection{An maximal estimate}
To end this section, we set up a $BL_A^{2,q}$ estimate which will be used in the interpolation argument in subsection \ref{s4.21}.

Let $1\leq m\leq n+1$. Let $Z$ denote the $m$ dimension variety $Z(P_1,...,P_{n+1-m})$ which consists of the common zeros of polynomials $P_1,...,P_{n+1-m}$ on $\mathbb{R}^{n+1}$. The variety $Z$ is a
transverse complete intersection in the sense,
$$\nabla P_1(x)\wedge...\wedge\nabla P_{n+1-m}(x)\neq 0\ \ \text{for all}\ x\in Z.$$
Given a small constant  $\delta_m>0$ depending on the dimension $m$. For any $(\theta,\nu)\in \mathbb T[R]$, we say that a tube $T_{\theta,\nu}$ is $R^{-1/2+\delta_m}$-tangent to $Z$ on $P_R$ if
$$T_{\theta,\nu}\subset N_{R^{1/2+\delta_m}}Z\cap P_R$$
and
$$\angle(G(\theta),T_zZ)\lesssim R^{-1/2+\delta_m}$$
for any $z\in N_{2R^{1/2+\delta_m}}(T_{\theta,\nu})\cap 2P_R\cap Z$. Here $T_zZ$ denotes the $m$-dimension tangent space of the variety $Z$ at $z$. As before
$$\mathbb{T}_Z(R):=\{(\theta,\nu)\in\mathbb{T}[R]:T_{\theta,\nu}\ \text{is}\ R^{-\frac{1}{2}+\delta_m}\ \text{-tangent to}\ Z\}.$$

The main result in this part is the following theorem.
\begin{theorem}\label{c3.1}
	Let $n\geq3$ and $2\leq m\leq n+1$. Suppose that $f\in L^2(\mathbb{R}^{n})$ is Fourier supported in $\mathbb{B}^n(Ne_1,1)$ and concentrated in $\mathbb{T}_Z(R)$. For any $\epsilon>0$,
	\begin{equation}\label{3.11}
		\|e^{it\Delta}f\|_{BL_A^{2,q}(P_R)}\lesssim_{\epsilon} K^{O(1)}R^{\frac{m}{2(m+1)}+\epsilon}\|f\|_{L^2}		
	\end{equation}
	hold for any $q>1/\epsilon^4$.
\end{theorem}

From Lemma \ref{l3.1}, it can be reduce to set up the following maximal estimate.

\begin{proposition}\label{p3.1}
	Let $2\leq m\leq n+1$. Suppose that $f\in L^2(\mathbb{R}^{n})$ is Fourier supported in $\mathbb{B}^n(Ne_1,1)$ and concentrated in $\mathbb{T}_Z(R)$. For any $\epsilon>0$ and $R\geq 1$, then
	\begin{equation}\label{e2.6}
		\|e^{it\Delta}f\|_{L^{2}_xL^{\infty}_t(P_R)}\lesssim_ \epsilon N^{\frac12} R^{\frac{m}{2(m+1)}+\epsilon} \|f\|_{L^2}.
	\end{equation}
\end{proposition}
The proof is  similar to  Proposition \ref{p2.3}.
\begin{proof}
	
	No loss of generality,  we assume $\|f\|_{L^2}=1$. In the same way, we introduce in the notation $$ A_H=\big\{x:(x,t)\in P_R,\sup_{0<t\leq R}|e^{it\Delta}f(x)|\sim H\big\}$$ and assume $R^{-100}\leq H\leq 1$. By the dyadic pigeonhole principle, we have
	\begin{equation}
		\|e^{it\Delta}f\|_{L_x^{2}L_t^{\infty}(P_R)}\lesssim (\log R)H|A_H|^{\frac12}.
	\end{equation}
	The Fourier transform of $e^{it\Delta}f$ is supported on a rectangular box of dimensions $$1\times\cdots\times 1\times N.$$
	Then it is essentially constant on a rectangular box of dimensions $1\times\cdots\times 1\times 1/N$. One can choose a collection $X\subset P_R$ is the union of the rectangular box of dimensions $1\times ... \times 1\times 1/N$ satisfying:
	\begin{itemize}
		\item Each vertical thin tube of dimensions $1\times ...\times 1\times R$ contains at most one rectangular box $1\times ...\times 1\times 1/N$ in $X$;
		\item The projection of set $X$ on $\mathbb{R}^n$ covers $A_H$.
	\end{itemize}
	By the definition of $X$, we have
	\begin{equation}\label{e2.9}
		|A_H|\sim N |X|.
	\end{equation}
	The parallelepiped $P_R$ is decomposed into $R^{\frac{1}{2}}$-parallelepipeds $P_j$ and  intersect $X$. By pigeonhole principle, we choose $P_j$s satisfying:
	\begin{itemize}
		\item $\|e^{it\Delta}f\|_{L^{\frac{2(m+1)}{(m-1)}}(P_j)}\sim \gamma\quad \text{for some dyadic number}\, R^{-C}\leq \gamma\leq R^C$;
		\item Each horizontal $R^{\frac12}$-slab  contains $\sim \sigma$ many $P_j$s.
	\end{itemize}
	Let $Y=\bigcup_{j=1}^{J}P_j$, we obtain
	\begin{equation}\label{e2.10}
		|X|\lesssim (\log R)^2|X\cap Y|\lesssim (\log R)^2 JR^{\frac n2}
	\end{equation}
	and $$J\lesssim \sigma R^{\frac12}.$$
	Using the locally constant property and the Galilean transform, we can obtain
	\begin{align}\label{eq:3.7}
		\|e^{it\Delta}f\|_{L^{2}_xL^{\infty}_t(P_R)}
		&\lesssim_\epsilon R^{\epsilon}N^{\frac12}|X|^{\frac12-\frac{m-1}{2(m+1)}}H|X\cap Y|^{\frac{m-1}{2(m+1)}}\notag\\
		& \lesssim_\epsilon R^{\epsilon}N^{\frac12}(\sigma R^{\frac{n+1}{2}})^{\frac{1}{m+1}}\|e^{it\Delta}f(x)\|_{L^{\frac{2(m+1)}{m-1}}(Y)}\notag\\
		& \lesssim_\epsilon R^{\epsilon}N^{\frac12}(\sigma R^{\frac{n+1}{2}})^{\frac{1}{m+1}}\|e^{i\tilde{t}\Delta}g(\tilde{x})\|_{L^{\frac{2(m+1)}{m-1}}(\tilde{Y})}.
	\end{align}
	Here $g$ is the image of $f$ under the Galilean transform satisfying $\supp \hat{g}\subset \mathbb{B}^n(0,1)$ and $\|g\|_{L^2}=1$. And  $\tilde{Y}=\bigcup_{j=1}^{J}Q_j\subset \mathbb{B}^{n+1}(0,R)$ is the image of $Y$. Since the tangent-to-variety condition is preserved under the Galilean transform, we finish the proof of Proposition \ref{p3.1} by recalling the follow linear refined Strichartz estimate.
\end{proof}

\begin{theorem}\cite{dglz18} (Linear refined Strichartz for $m$-variety in dimension $n+1$). 
	Let $2\leq m\leq n+1$ and $p_m=2(m+1)/(m-1)$.
	Suppose that $Z=Z(P_1,...,P_{n+1-m})$ is a transverse complete intersection where
	$Deg P_i\leq R^{\delta_{i}}$. Here $\delta_{i}\ll\delta$ is a small parameter. Suppose that $g\in L^2(\mathbb{R}^n)$ is Fourier supported in $\mathbb{B}^n(0,1)$ and concentrated in $\mathbb{T}_Z(R)$ on $\mathbb{B}^{n+1}(0,R)$. Suppose that $Q_1,\ Q_2,...$ are lattice $R^{1/2}$-cubes in $\mathbb{B}^{n+1}(0,R)$, so that
	$$\|e^{it\Delta}g\|_{L^{p_m}(Q_j)}\ \ \text{is essentially constant in}\ j.$$
	Suppose that these cubes are arranged in horizontal slabs of the form $\mathbb{R}\times...\times\mathbb{R}\times\{t_0,t_0\}$, and that each such slab contains $\sim\sigma$ cubes $Q_j$. Let $Y$ denote $\bigcup_j Q_j$. Then
	$$\|e^{it\Delta}g\|_{L^{p_m}(Y)}\leq C_{\epsilon}R^{\epsilon}R^{-(n+1-m)/2(m+1)}\sigma^{-\frac{1}{m+1}}\|g\|_{L^2}.$$
\end{theorem}

\section{Algorithm}\label{Alg}

We now give the proof of Theorem \ref{t2.3x}. If $M\geq R^{\frac12-\delta}$, by Lemma \ref{l3.1}, we have $$\|e^{it\Delta}f\|_{BL_A^{p,\infty}(P_R)}\lesssim_{p, \epsilon} M^{-\epsilon^2}R^{\epsilon}\|f\|_{L^2}$$
holds for $p>\frac{2(n+1)}{n}$. So we can assume $M< R^{\frac12-\delta}$. To show that, for supp$\hat{f}\subset \mathbb{B}^n(\xi_0, \frac{1}{M})$,
\begin{equation*}
		\|e^{it\Delta}f\|_{BL_A^{p,\infty}(P_R)}\lesssim_{p,\epsilon} M^{-\epsilon^2}R^{\epsilon}\|f\|_{L^2}
\end{equation*}
holds for $p\geq\tilde{p}_n$. We only need to prove that, for supp$\hat{f}\subset \mathbb{B}^n(Ne_1, 1)$, 
\begin{equation}\label{e3.14}
	\|e^{it\Delta}f\|_{BL_A^{p,\infty}(P_R)}\lesssim_{p,\epsilon} R^{\epsilon}\|f\|_{L^2}
\end{equation}
holds for $p\geq \tilde{p}_n$ since 
$$R^{\frac{\epsilon}{2}}\leq R^{\epsilon}R^{-\frac{\epsilon^2}{2}}\leq R^{\epsilon}M^{-\epsilon^2}.$$ For \eqref{e3.14}, it is suffice to prove that
\begin{equation}\label{e3.13}
	\|e^{it\Delta}f\|_{BL_A^{p,q}(P_R)}\lesssim_{p,\epsilon} R^{\epsilon}\|f\|_{L^2}
\end{equation}
holds for $p\geq\tilde{p}_n$ and any $q>1/\epsilon^4$.

We will adopt two algorithms which were introduced by Hickman and Rogers \cite{hg19} and developed by Guth \cite{G18} in the study of Fourier restriction conjecture. The algorithms conclude the scale induction and the dimensional induction. We refer the readers \cite{G18,hg19,cmw21} for more details.

\begin{proposition}\cite{G18}\label{p3.4}
	Let $1\leq r^{\frac12}\leq \rho\leq r$ and $P_{r}\cap N_{{r}^{1/2+\delta_m}}Z\neq\emptyset$. Here $m$ denotes the dimension of the variety $Z$ and $\delta_m$ is a small number depends $m$. $f$ is concentrated in $\mathbb{T}_Z(r)$. Then there is a set of translates $\mathcal{B}\subset \mathbb{B}^{n+1}(0,{r}^{1/2+\delta_m})$ such that
	$$\|e^{it\Delta}f\|_{BL_A^{p,q}(P_r\cap N_{r^{1/2+\delta_m}}Z)}^p\lesssim \log^2r\sum_{b\in\mathcal{B}}\|e^{it\Delta}\tilde{f}_b\|_{BL_A^{p,q}(P_{\rho}\cap N_{{\rho}^{1/2+\delta_m}}(Z+b))}^p+ r^{-M}\|f\|_{L^2}^p$$
	where $\tilde{f}_b$ is concentrated in $\mathbb{T}_{Z+b}(\rho)$ with $Z+b=\{z+b:z\in Z\}$.
	Moreover, 
	$$\sum_{b\in\mathcal{B}}\|\tilde{f}_b\|_{L^2}^2\lesssim\|f\|_{L^2}^2$$
	and for any $b\in\mathcal{B}$,
	$$\|\tilde{f}_b\|_{L^2}^2\lesssim r^{O(\delta_m)}\big(\frac {r}{\rho}\big)^{-\frac{n+1-m}{2}}\|f\|_{L^2}^2.$$
\end{proposition}

\begin{theorem}\cite{G18,cmw21}\label{t3.2}
	Fix $r\gg 1$ and $1\leq p,q<\infty$. Suppose that $Z$ is an $m$-dimensional transverse complete intersection of degree at most $D_Z$. Suppose that $f$ is concentrated in $\mathbb{T}_Z(r)$. Then there exists a
	constant $D=D(\epsilon,D_Z)$ such that at least one of the following cases holds.\\
\textbf{Cellular case}: there exist $O(D^m)$ shrunken cells $O_i$  with size at most $r/2$ such that for each $i$
		$$\|e^{it\Delta}f\|_{BL_A^{p,q}(P_r\cap N_{r^{1/2+\delta_m}}Z)}^p\lesssim D^m\|e^{it\Delta}f\|_{BL_A^{p,q}(O_i)}^p.$$
		\textbf{Algebraic case}: there exists an $m-1$-dimensional transverse complete intersection $Y$ of degree at most $O(D)$ such that
		$$\|e^{it\Delta}f\|_{BL_A^{p,q}(P_r\cap N_{r^{1/2+\delta_m}}Z)}^p \lesssim \|e^{it\Delta}f\|_{BL_A^{p,q}(P_r\cap N_{r^{1/2+\delta_m}}Y)}^p.$$
\end{theorem}

\subsection{Sketch of the first algorithm} Suppose function $f$ satisfing the conditions of Theorem \ref{t3.2}, we consider the cellular case and the algebraic case, respectively.

If the cellular case dominates, 
	$$\|e^{it\Delta}f\|_{BL_A^{p,q}(P_r\cap N_{r^{1/2+\delta_m}}Z)}^p\lesssim D^m\|e^{it\Delta}f\|_{BL_A^{p,q}(O_i)}^p.$$
For each $i$, set
$$f_{O_i}=\sum_{(\theta,\nu)\in \mathbb{T}_Z(r),T_{\theta,\nu}\cap O_i\neq\emptyset}f_{\theta,\nu}.$$
We have
$$\|e^{it\Delta}f\|_{BL_A^{p,q}(P_r\cap N_{r^{1/2+\delta_m}}Z)}^p\lesssim \sum_{O_i}\|e^{it\Delta}f_{O_i}\|_{BL_A^{p,q}(O_i)}^p+\text{RapDec}(r)\|f\|_{L^2}^p,$$
$$\sum_{O_i}\|f_{O_i}\|_{L^2}^2\lesssim D\|f\|_{L^2}^2$$
and 
$$\|f_{O_i}\|_{L^2}^2\lesssim D^{-(m-1)}\|f\|_{L^2}^2,\ \text{for each}\ O_i.$$
So far the estimates are reduced from the $r$-size set $P_r$ to the $r/2$-size set $Q_i$. One can continue the argument to $f_{O_i}$ because $f_{O_i}$ concentrated in $\mathbb{T}_Z(r/2)$ (see Lemma \ref{l2.5}).

If the algebraic case dominates, there exists an $m-1$-dimensional transverse complete intersection $Y$ of degree at most $O(D)$ such that
$$\|e^{it\Delta}f\|_{BL_A^{p,q}(P_r\cap N_{r^{1/2+\delta_m}}Z)}^p \lesssim \|e^{it\Delta}f\|_{BL_A^{p,q}(P_r\cap N_{r^{1/2+\delta_m}}Y)}^p.$$
Even the dimension of the variety $Y$ is already $m-1$, but the function $f$ may not tangent to $Y$. 
Let $\mathcal{P}$ be a cover of $P_r\cap N_{r^{1/2+\delta_m}}Y$ by finitely-overlapping $P_{r_1}$ with some suitable chosen $r_1<r$. For each $P_{r_1}\in\mathcal{P}$, let 
$$\mathbb{T}_{P_{r_1}}=\{(\theta,\nu)\in \mathbb{T}_Z(r):T_{\theta,\nu}\cap P_{r_1}\cap N_{r^{1/2+\delta_m}}Y\neq\emptyset\}.$$
A tube $T_{\theta,\nu}$ is $r_1^{1/2+\delta_{m-1}}$-tangent to $Y$ on $P_{r_1}$ means
$$T_{\theta,\nu}\cap 2P_{r_1}\subset N_{2r_1^{1/2+\delta_{m-1}}}Y =N_{2r^{1/2+\delta_m}}Y $$
and for any $y\in Y\cap 2P_{r_1}\cap N_{r_1^{1/2+\delta_{m-1}}}(T_{\theta,\nu})$,
$$\angle(G(\theta),T_yY)\lesssim r_1^{-1/2+\delta_{m-1}}.$$
Here we used a fact that $r_1^{\frac12+\delta_{m-1}}=r^{\frac12+\delta_m}.$ Which can be obtained by suitable choose of $\delta_{m-1}$.
Denote 
$$\mathbb{T}_{P_{r_1},tang}=\{(\theta,\nu)\in\mathbb{T}_{P_{r_1}}: T_{\theta,\nu}\ \text{is}\ r_1^{1/2+\delta_{m-1}}-\text{tangency to}\ Y \text{on}\ P_{r_1} \}$$
and $\mathbb{T}_{P_{r_1},trans}=\mathbb{T}_{P_{r_1}}\setminus\mathbb{T}_{P_{r_1},tang}$, we have
\begin{align*}
	&\|e^{it\Delta}f\|_{BL_A^{p,q}(P_r\cap N_{r^{1/2+\delta_m}}Y)}^p\\
	& \lesssim \sum_{P_{r_1}\in\mathcal{P}}\|e^{it\Delta}f_{P_{r_1},tang}\|_{BL_{A/2}^{p,q}(P_{r_1})}^p+\sum_{P_{r_1}\in\mathcal{P}}\|e^{it\Delta}f_{P_{r_1},trans}\|_{BL_{A/2}^{p,q}(P_{r_1})}^p+\text{RapDec}(r)\|f\|_{L^2}^p,
\end{align*}
where
$$f_{P_{r_1},tang}=\sum_{(\theta,\nu)\in \mathbb{T}_{P_{r_1},tang}} f_{\theta,\nu}\ \ \text{and}\ \ f_{P_{r_1},trans}=\sum_{(\theta,\nu)\in \mathbb{T}_{P_{r_1},trans}} f_{\theta,\nu}.$$
By the pigeonhole principle, if $\sum_{P_{r_1}\in\mathcal{P}}\|e^{it\Delta}f_{P_{r_1},tang}\|_{BL_{A/2}^{p,q}(P_{r_1})}^p$ dominates, by Lemma \ref{l2.5}, we have
$$\sum_{P_{r_1}\in\mathcal{P}}\|e^{it\Delta}f_{P_{r_1},tang}\|_{BL_{A/2}^{p,q}(P_{r_1})}^p\lesssim \sum_{O_j}\|e^{it\Delta}f_{O_j}\|_{BL_{A/2}^{p,q}(O_j)}^p, $$
where $f_{O_j}$ is concentrated in $\mathbb{T}_Y(r_1)$ and the size of $O_j$ is $r_1$. We obtain  the dimensional reducing aim and with the lost of some constant independent of $D$, $r$, and $m$.

Otherwise, we face the most difficult case
$$\|e^{it\Delta}f\|_{BL_A^{p,q}(P_r)}^p\lesssim \sum_{P_{r_1}\in\mathcal{P}}\|e^{it\Delta}f_{P_{r_1},trans}\|_{BL_{A/2}^{p,q}(P_{r_1})}^p.$$
By Propositions \ref{p3.4}, we have the following estimates:
$$\sum_{P_{r_1}\in\mathcal{P}}\|e^{it\Delta}f_{P_{r_1},trans}\|_{BL_{A/2}^{p,q}(P_{r_1})}^p\lesssim \log^2r\sum_{O_j}\|e^{it\Delta}f_{O_j}\|_{BL_{A/2}^{p,q}(O_j)}^p,$$
$$\sum_{O_j}\|f_{O_j}\|_{L^2}^2\lesssim \|f\|_{L^2}^2$$
and for each $O_j$,
$$\|f_{O_j}\|_{L^2}^2\lesssim r^{O(\delta_m)}\Big(\frac{r}{r_1}\Big)^{-\frac{n+1-m}{2}}\|f\|_{L^2}^2,$$
where $f_{O_j}$ is concentrated in $\mathbb{T}_Z(r_1)$ with $b\in \mathbb{B}^{n+1}(0,{r}^{1/2+\delta_m})$ and the size of $O_j$ at most $r_1$. Next, for each $f_{O_j}$, we proceed to apply Theorem \ref{t3.2}.

\subsection{The first algorithm 
[alg 1]} Let $p\geq 2$, $0<\epsilon\ll 1$ and a family of small parameters
$$\epsilon^C\leq \delta\ll\delta_n\ll \delta_{n-1}\ll \dots\ll \delta_1\ll\delta_0\ll\epsilon.$$
\\
\fbox{Input}
\begin{itemize}
	\item A parallelepiped $P_r$ with size $r\gg 1$.
	\item An admissible large integer $A\in \mathbb{N}$.
	\item A function $f$ with supp$\hat{f}\subset \mathbb{B}^n(Ne_1,1)$ and concentrated on wave packets which are $r^{-\frac 12+\delta_m}$-tangent to $Z$ in $P_r$ with a transverse complete intersection $Z$ of dimension $m\geq 2$.
\end{itemize}
\fbox{Output} The $j$th stage of [alg 1] outputs:
\begin{itemize}
	\item A choice of spatial scale $\rho_j\geq 1$.
	\item Certain integers $\sharp a(j),\ \sharp c(j)\in\mathbb{N}_0$ satisfying $\sharp a(j)+\sharp c(j)=j$, where the integers $\sharp a(j),\ \sharp c(j)$ indicate the number of occurrences of the algebraic cases and the cellular cases, respectively.
	\item  A family of subsets $\mathcal{O}_j$ of $\mathbb{R}^{n+1}$ which will be referred to as cells. Each cell $O_j\in\mathcal{O}_j$ has size at most $\rho_j$.
	\item   A collection of functions $\{f_{O_j}\}_{O_j\in\mathcal{O}_j}$. Each $f_{O_j}$ is concentrated on wave packets of  scale $\rho_j$ which are $\rho_j^{-\frac12+\delta_m}$ -tangent to some translation of $Z$ on $O_j$.
	\item A large integer $d\in\mathbb{N}$ which depends only on the admissible parameters and deg$Z$.
\end{itemize}

With these outputs, we denote
$$C_{j,\delta}^{I}(d,r)=d^{\sharp c(j)\delta}(\log r)^{2\sharp a(j)(1+\delta)},$$
$$C_{j,\delta}^{II}(d)=d^{\sharp c(j)\delta+(n+1)\sharp a(j)(1+\delta)},$$
$$C_{j,\delta}^{III}(d,r)=d^{\sharp c(j)\delta+\sharp a(j)\delta}r^{\tilde{C}\sharp a(j)\delta_m}.$$ It is clear that
\begin{equation}\label{e3.12}
	C_{j,\delta}^{I}(d,r),\,C_{j,\delta}^{II}(d),\,C_{j,\delta}^{III}(d,r)\lesssim_{d,\delta}r^{\delta_0}d^{\sharp c(j)\delta}.
\end{equation}
Let $A_j=2^{-\sharp a(j)}A\in\mathbb{N}$. We obtain the decomposition which satisfy the following properties.\\
Property 1. Most of the mass of $\|e^{it\Delta}f\|_{BL_A^{p,q}(P_r)}^p$ is concentrated on the $O_j\in\mathcal{O}_j$, i.e.
$$\|e^{it\Delta}f\|_{BL_A^{p,q}(P_r)}^p\leq C_{j,\delta}^{I}(d,r)\sum_{O_j\in\mathcal{O}_j} \|e^{it\Delta}f_{O_j}\|_{BL_{A_j}^{p,q}(O_j)}^p+ \text{RapDec}(r)\|f\|_{L^2}^p.$$
Property 2. The functions $f_{O_j}$ satisfy
$$\sum_{O_j\in\mathcal{O}_j}\|f_{O_j}\|_{L^2}^2\leq C_{j,\delta}^{II}(d)d^{\sharp c(j)}\|f\|_{L^2}^2.$$
Property 3. Each individual $f_{O_j}$ satisfies
$$\|f_{O_j}\|_{L^2}^2\leq C_{j,\delta}^{III}(d,r)\Big(\frac{r}{\rho_j}\Big)^{-\frac{n+1-m}{2}}d^{-\sharp c(j)(m-1)}\|f\|_{L^2}^2.$$
\fbox{Stopping conditions} Suppose that we are at the $j$th stage of [alg 1]. We stop our algorithm if one of [tiny] and [tang] holds true:
 
[tiny] The algorithm terminates if $\rho_j\leq r^{\tilde{\delta}_{m-1}}$. 

[tang] The algorithm terminates if
$$\sum_{O_j\in\mathcal{O}_j} \|e^{it\Delta}f_{O_j}\|_{BL_{A_j}^{p,q}(O_j)}^p\leq C_{tang} \sum_{S\in\mathcal{S}} \|e^{it\Delta}f_{S}\|_{BL_{A_j/2}^{p,q}(P_{\tilde{\rho}})}^p$$
and 
$$\sum_{S\in\mathcal{S}}\|f_{S}\|_{L^2}^2\leq C_{tang}r^{(n+1)\tilde{\delta}_{m-1}}\sum_{O_j\in\mathcal{O}_j}\|f_{O_j}\|_{L^2}^2$$
hold for $(1-\tilde{\delta}_{m-1})(1/2+\delta_{m-1})=(1/2+\delta_m)$, $\tilde{\rho}=\rho_j^{1-\tilde{\delta}_{m-1}}$, and for some fixed dimensional constant $C_{tang}$. Here $\mathcal{S}$ is a collection of
$(m-1)$-dimensional transverse complete intersections in $\mathbb{R}^{n+1}$ all of degree $O(d)$. $f_{S}$ is a function which is $\tilde{\rho}^{-1/2+\delta_{m-1}}$-tangent to $S\in\mathcal{S}$ on $P_{\tilde{\rho}}$.

By repeatedly applying the algorithm [alg 1], we have the following algorithm [alg 2]

\subsection{The second algorithm [alg 2]}\label{s4.21} Let $p_{n+1},\dots,p_{k}$ denote Lebesgue exponents to be fixed later and satisfy
$$2\leq p:=p_{n+1}\leq...\leq p_{k+1}\leq p_k.$$
Let $0\leq \alpha_l, \,\beta_l\leq 1$ be defined by
$$\frac{1}{p_l}=\frac{1-\alpha_{l-1}}{2}+\frac{\alpha_{l-1}}{p_{l-1}}\ \text{and}\ \beta_l=\prod_{i=l}^n\alpha_i=\Big(\frac12-\frac{1}{p_l}\Big)^{-1}\Big(\frac12-\frac{1}{p_n}\Big)\ \ \text{for}\ k+1\leq l\leq n,$$
and $\alpha_{n+1}=\beta_{n+1}=1$.\\
\fbox{Input} Fix $1\ll A\ll R$. Function $f$ satisfies $supp\hat{f}\subset \mathbb{B}^n(Ne_1,1)$ and the non-degeneracy hypothesis
$$\|e^{it\Delta}f\|_{BL_A^{p,q}(P_R)}\geq CR^{\epsilon}\|f\|_{L^2}.$$
\fbox{Output} The $(n+2-l)$-th step of the recursion will produce:
\begin{itemize}
	\item An $(n+2-l)$-tuple of scales $\vec{r}=(r_{n+1},...,r_l)$ satisfying $R=r_{n+1}>r_{n}>...>r_l$; large parameters $\vec{D}=(D_{n+1},...,D_{l})$, where $D_i(l\leq i\leq n+1)$ are defined by
	the upper bound of $O(r_{i+1}^{\delta_0}d^{\delta\log r_{i+1}})$; integers $\vec{A}=(A_{n+1},\dots,A_l)$ satisfying $A=A_{n+1}>A_{n}>\dots>A_l$.
	\item  A family transverse complete intersections $\mathcal{S}_l$. Each $S_i\in\mathcal{S}_i(l\leq i\leq n+1)$ satisfies dim$S_i=i$.
	\item  A function $f_{S_l}$ associated with each $S_l\in\mathcal{S}_l$ is concentrated on scale $r_l$ wave packets which are $r_l^{-1/2+\delta_l}$-tangent to $S_l$ in $P_{r_l}$.
\end{itemize}

These data satisfy the following properties.\\
Property 1. \begin{equation}\label{e3.8}
	\|e^{it\Delta}f\|_{BL_A^{p,q}(P_R)}\lesssim M(\vec{r_l},\vec{D_l})R^{O(\delta_0)}\|f\|_{L^2}^{1-\beta_l}\Big(\sum_{S_l\in\mathcal{S}_l}\|e^{it\Delta}f_{S_l}\|_{BL_{A_l}^{p_l,q}(P_{r_l})}^{p_l}\Big)^{\frac{\beta_l}{p_l}},
\end{equation}
where 
$$M(\vec{r_l},\vec{D_l}):= \Big(\prod_{i=l}^{n}D_i\Big)^{(n+1-l)\delta}\Big(\prod_{i=l}^{n}r_i^{\frac{i}{2(i+1)}(\beta_{i+1}-\beta_i)}D_i^{\frac12(\beta_{i+1}-\beta_l)}\Big).$$
Propety 2. For $l\leq n$,
\begin{equation*}
	\sum_{S_l\in\mathcal{S}_l}\|f_{S_l}\|_{L^2}^2\lesssim D_l^{1+\delta} R^{O(\delta_0)}\sum_{S_{l+1}\in\mathcal{S}_{l+1}}\|f_{S_{l+1}}\|_{L^2}^2.
\end{equation*}
Property 3. For $l\leq n$, 
\begin{equation*}
	\max_{S_l\in\mathcal{S}_l}\|f_{S_l}\|_{L^2}^2\lesssim \Big(\frac{r_{l+1}}{r_l}\Big)^{-\frac{n-l}{2}}D_l^{-l+\delta}R^{O(\delta_0)}\max_{S_{l+1}\in\mathcal{S}_{l+1}}\|f_{S_{l+1}}\|_{L^2}^2.
\end{equation*}
\fbox{Stopping condition} The recursive stage of [alg 2] terminates if he following [tiny-dom] happens.

[tiny-dom] Suppose that the inequality
$$\sum_{S_l\in\mathcal{S}_l}\|e^{it\Delta}f_{S_l}\|_{BL_{A_l}^{p_l,q}(P_{r_l})}^{p_l}\leq 2\sum_{S_l\in\mathcal{S}_{l,tiny}}\|e^{it\Delta}f_{S_l}\|_{BL_{A_l/2}^{p_l,q}(P_{r_l})}^{p_l}$$
holds, where the right-hand summation is restricted to those $S_l\in\mathcal{S}_l$ for which [alg 1] terminates owing to the stopping condition [tiny].

We now prove Properties 1  of [alg 2]  which relies on Theorem \ref{c3.1}. Properties 2 and 3 can be obtained by repeatedly applying of Properties 2 and 3 in [alg 1]. 

Assume that the condition [tiny-dom] does not happen. Necessarily,
$$\sum_{S_l\in\mathcal{S}_l}\|e^{it\Delta}f_{S_l}\|_{BL_{A_l}^{p_l,q}(P_{r_l})}^{p_l}\leq 2\sum_{S_l\in\mathcal{S}_{l,tang}}\|e^{it\Delta}f_{S_l}\|_{BL_{A_l/2}^{p_l,q}(P_{r_l})}^{p_l},$$
where the right-hand summation is restricted to those $S_l\in\mathcal{S}_l$ for which [alg 1] terminates owing to [tang]. 
For each $S_l\in \mathcal{S}_{l,tang}$, let $\mathcal{S}_{l-1}[S_l]$ denote a collection of transverse complete intersections in $\mathbb{R}^{n+1}$ of dimension $l-1$ which is produced by the stopping condition [tang] in [alg 1] on $\|e^{it\Delta}f_{S_l}\|_{BL_A^{p_l,q}(P_{r_l})}^{p_l}$. Then
$$\|e^{it\Delta}f_{S_l}\|_{BL_{A_l}^{p_l,q}(P_{r_l})}^{p_l}\lesssim R^{O(\delta_0)}D_{l-1}^{\delta}\sum_{S_{l-1}\in \mathcal{S}_{l-1}[S_l]}\|e^{it\Delta}f_{S_{l-1}}\|_{BL_{2A_{l-1}}^{p_l,q}(P_{r_{l-1}})}^{p_l}.$$
Let $\mathcal{S}_{l-1}$ denote the structured set 
$$\mathcal{S}_{l-1}=\{S_{l-1}:S_l\in \mathcal{S}_{l,tang}\ \text{and}\ S_{l-1}\in \mathcal{S}_{l-1}[S_l] \}.$$
By induction on $l$, we have
$$\|e^{it\Delta}f\|_{BL_A^{p,q}(P_R)}\lesssim D_{l-1}^{\delta} M(\vec{r_l},\vec{D_l})R^{O(\delta_0)}\|f\|_{L^2}^{1-\beta_l}\Big(\sum_{S_{l-1}\in\mathcal{S}_{l-1}}\|e^{it\Delta}f_{S_{l-1}}\|_{BL_{2A_{l-1}}^{p_l,q}(P_{r_{l-1}})}^{p_l}\Big)^{\frac{\beta_l}{p_l}}.$$
By Proposition \ref{p3.2} (iii) and H\"older's inequality, we have 
\begin{align*}
&\Big(\sum_{S_{l-1}\in\mathcal{S}_{l-1}}\|e^{it\Delta}f_{S_{l-1}}\|_{BL_{2A_{l-1}}^{p_l,q}(P_{r_{l-1}})}^{p_l}\Big)^{\frac{1}{p_l}}\\
	& \leq \Big\| \|e^{it\Delta}f_{S_{l-1}}\|_{BL_{A_{l-1}}^{p_{l-1},q}(P_{r_{l-1}})}  \Big\|_{l^{p_{l-1}}(\mathcal{S}_{l-1})}^{\alpha_{l-1}}\Big\| \|e^{it\Delta}f_{S_{l-1}}\|_{BL_{A_{l-1}}^{2,q}(P_{r_{l-1}})}  \Big\|_{l^{2}(\mathcal{S}_{l-1})}^{1-\alpha_{l-1}}.
\end{align*}
Combining Theorem \ref{c3.1} and Property 2, it follows that
$$\Big\| \|e^{it\Delta}f_{S_{l-1}}\|_{BL_{A_{l-1}}^{2,q}(P_{r_{l-1}})}  \Big\|_{l^{2}(\mathcal{S}_{l-1})}\lesssim r_{l-1}^{\frac{l-1}{2l}}(\prod_{i=l-1}^nD_i^{1+\delta})^{\frac 12}R^{O(\delta_0)}\|f\|_{L^2}.$$
Thus, we obtain 
\begin{align*}
	&\|e^{it\Delta}f\|_{BL_A^{p,q}(P_R)}\\
	& \lesssim  M(\vec{r}_{l-1},\vec{D}_{l-1})R^{O(\delta_0)}\|f\|_{L^2}^{1-\beta_{l-1}}\Big(\sum_{S_{l-1}\in\mathcal{S}_{l-1}}\|e^{it\Delta}f_{S_{l-1}}\|_{BL_{A_{l-1}}^{p_{l-1},q}(P_{r_{l-1}})}^{p_{l-1}}\Big)^{\frac{\beta_{l-1}}{p_{l-1}}}.
\end{align*}
\fbox{The final stage} The algorithm necessarily terminates at the $n$-th step. In fact, if \eqref{e3.8} would hold for $l=1$ and functions $f_{S_1}$ concentrated on wave packets which are tangent to some transverse complete intersections of dimension $1$. By the vanishing property of the $2$-broad norm which can be easily proven, we have 
\begin{equation}\label{e4.12}
	\|e^{it\Delta}f_{S_{1}}\|_{BL_A^{p_{1},q}(P_{r_{1}})}\lesssim \text{RapDec}(R)\|f\|_{L^2}.
\end{equation}
 and hence $\|e^{it\Delta}f\|_{BL_A^{p,q}(P_R)}\lesssim\text{RapDec}(R)\|f\|_{L^2}$ by \eqref{e3.8}. This contradicts the non-degeneracy hypothesis.

In fact, from the definition of $BL_A^{p,q}$ and choosing the large number $K\ll r_1^{1/2}$. Let $D_K\times I_{K/N}\subset N_{r_1^{1/2+\delta_1}}(S_1)\cap P_{r_1}$ and $V=T_{z_0}S_1$ with $z_0\in S_1\cap N_{r_1^{1/2+\delta_1}}(D_K\times I_{K/N})$. Then $\dim V=1$. For each $(\theta,\nu)\in \mathbb{T}_{S_1}(r_1)$ and $T_{\theta,\nu}\cap(D_K\times I_{K/N})\neq \emptyset$, we have 
	$$\angle(G(\theta),V)\lesssim r_1^{-\frac12+\delta_1}.$$
It implies that for any $\theta\subset \tau$,
$$\angle(G(\tau),V)\leq K^{-1}.$$ So such $\tau$’s do not contribute to $\mu_{e^{it\Delta}f_{S_1}}(D_K\times I_{K/N})$ and hence obtain $\eqref{e4.12}$.

Suppose the recursion process terminates at step $n+2-k$ for some $k\geq2$. For each $S_k\in\mathcal{S}_{k,tiny}$, let $\mathcal{O}[S_k]$ denote the final collection of cells outputted by [alg 1]. Each $O\in\mathcal{O}[S_k]$
has the size at most $R^{\delta_0}$ by the stopping condition [tiny]. Let
$$\mathcal{O}=\bigcup\{\mathcal{O}[S_k]:S_k\in\mathcal{S}_{k,tiny}\}.$$
 From Properties 1, 2 and 3 of [alg 1] and [alg 2], it follows that the following three properties hold.

\textbf{Property 1}
\begin{equation}\label{e3.3}
	\|e^{it\Delta}f\|_{BL_A^{p,q}(P_R)}\lesssim M(\vec{r}_k,\vec{D}_k)R^{O(\delta_0)}\|f\|_{L^2}^{1-\beta_k}\Big(\sum_{O\in\mathcal{O}}\|e^{it\Delta}f_{O}\|_{BL_{A_{k-1}}^{p_k,q}(O)}^{p_k}\Big)^{\frac{\beta_k}{p_k}},
\end{equation}
where 
$$M(\vec{r}_k,\vec{D}_k):= \Big(\prod_{i=k}^{n}D_i\Big)^{(n+1-l)\delta}\Big(\prod_{i=k}^{n}r_i^{\frac{i}{2(i+1)}(\beta_{i+1}-\beta_i)}D_i^{\frac12(\beta_{i+1}-\beta_k)}\Big).$$

\textbf{Property 2} For $k\leq n+1$,
\begin{equation}\label{e2.4}
	\sum_{O\in\mathcal{O}}\|f_O\|_{L^2}^2\lesssim\Big(\prod_{i=k-1}^{n}D_i^{1+\delta}\Big)R^{O(\delta_0)}\|f\|_{L^2}^2.
\end{equation}

\textbf{Property 3} For $k\leq n+1$,
\begin{equation}\label{e3.5}
	\max_{O\in\mathcal{O}}\|f_O\|_{L^2}^2\lesssim\prod_{i=k-1}^{n}r_i^{-\frac12}D_i^{-i+\delta}R^{O(\delta_0)}\|f\|_{L^2}^2.
\end{equation}

Since each $O\in\mathcal{O}$ has the size at most $R^{\delta_0}$, we have 
$$\|e^{it\Delta}f_{O}\|_{BL_A^{p_l,q}(O)}\lesssim R^{\delta_0}\|f_{O}\|_{L^2}.$$
Combining this estimate and the above properties, one concludes
$$\|e^{it\Delta}f\|_{BL_A^{p,q}(P_R)}\lesssim  \prod_{i=k-1}^{n}r_i^{X_i}D_i^{Y_i+O(\delta)}R^{O(\delta_0)}\|f\|_{L^2}, $$
where
$$X_i:=\frac{i}{2(i+1)}(\beta_{i+1}-\beta_i)-\frac12(\frac12-\frac{1}{p_{n+1}})$$
and 
$$Y_i:=\frac{\beta_{i+1}}{2}-(i+1)(\frac12-\frac{1}{p_{n+1}}).$$

One easily verifies that 
\begin{equation}\label{e3.9}
	X_i\leq 0 \iff \Big(\frac12-\frac{1}{p_{i+1}}\Big)^{-1}-\Big(\frac12-\frac{1}{p_i}\Big)^{-1}\leq \frac{i+1}{i}
\end{equation}
and 
\begin{equation}\label{e3.7}
	Y_{i-1}\leq 0 \iff \Big(\frac12-\frac{1}{p_i}\Big)^{-1}-2i\leq 0.
\end{equation}

We start with $$\Big(\frac12-\frac{1}{p_k}\Big)^{-1}=2k.$$ By \eqref{e3.9} and  \eqref{e3.7}, we have 
$$\Big(\frac12-\frac{1}{p_{k+1}}\Big)^{-1}\leq \min\{2(k+1),\ 2k+\frac{k+1}{k}\}. $$
We choose $p_{k+1}$ such that
$$\Big(\frac12-\frac{1}{p_{k+1}}\Big)^{-1}=2k+\frac{k+1}{k}.$$
Repeating the above argument, we obtain
$$\Big(\frac12-\frac{1}{p_{n+1}}\Big)^{-1}=2k+\frac{k+1}{k}+...+\frac{n+1}{n}$$
and hence
$$p_{n+1}=2+\frac{4}{n+k-1+s_{k}^{n}}.$$
The worst case happens when $k=2$, i.e.,
$$\|e^{it\Delta}f\|_{BL_A^{p,q}(P_R)}\lesssim R^{\epsilon}\|f\|_{L^2},\  \text{when}\  p\geq \tilde{p}_n=p_{n+1}=2+\frac{4}{n+1+s_2^n}.$$
And hence we complete the proof of Theorem \ref{t4.6}.

\section{The proof of Theorem \ref{t1.1} }\label{sub.1}

Let $n\geq2$, we need to find $$p_n\in\big[\frac{2(n+1)}{n},\frac{2(n+3)}{(n+1)}\big)\ \ \text{and}\ \  \frac{n+1}{p_n}+\frac{1}{r_n}=\frac{n}{2}$$ such that
\begin{equation}\label{main estimate1}
	\||e^{it\Delta}f_1e^{it\Delta}f_2|^{\frac12}\|_{L_x^{p_n}L_t^{r_n}(P_R)}\lesssim_\epsilon N^{\frac{1}{p_n}-\frac{1}{r_n}}R^{\epsilon}||f_1||_{L^2}^\frac12||f_2||_{L^2}^\frac12.
\end{equation}
Here $\hat{f_1},\ \hat{f_2}$ are supported on subsets of $\mathbb{B}^n(Ne_1,1)$ with $N\geq 1$ and $\text{dist}(\text{supp}\hat{f_1},\text{supp}\hat{f_2})\geq \frac{1}{2}$.

In fact, by \eqref{main estimate1} and Proposition \ref{p2.2}, we have
\begin{equation}\label{main estimate}
	\||e^{it\Delta}f_1e^{it\Delta}f_2|^{\frac12}\|_{L_x^{p}L_t^{r}(\mathbb R^{n+1})}\lesssim  N^{\frac{1}{p}-\frac{1}{r}}||f_1||_{L^2}^\frac12||f_2||_{L^2}^\frac12
\end{equation}
for $p>p_n$ and $r>r_n$.
On the other hand, using Bernstein's inequality and the sharp bilinear estimate  established by Tao \cite{t03}, we have, for $r\geq p>\frac{2(n+3)}{n+1}$, 
\begin{align}\label{e2.5}
	\||e^{it\Delta}f_1 e^{it\Delta}f_2|^\frac12\|_{L_x^{p}L_t^{r}(\mathbb{R}^{n+1})}
	& \lesssim N^{\frac{1}{p}-\frac{1}{r}}\||e^{it\Delta}f_1 e^{it\Delta}f_2|^{\frac12}\|_{L_{x,t}^{p}(\mathbb{R}^{n+1})}\nonumber\\
	&\lesssim N^{\frac{1}{p}-\frac{1}{r}}\|f_1\|_{L^2}^\frac12\|f_2\|_{L^2}^\frac12.
\end{align}
Interpolating \eqref{main estimate} with \eqref{e2.5}, for each critical  pair $(\tilde p,\tilde r)$ with $p_n<\tilde p\leq \frac{2(n+2)}{n}$, we can find a neighborhood $\mathcal{N}(\tilde p,\tilde r)$ such that  
\begin{equation}\label{eq:1.11}
	\||e^{it\Delta}f_1 e^{it\Delta}f_2|^\frac12\|_{L_x^{p}L_t^{r}(\mathbb{R}^{n+1})}\lesssim N^{\frac{1}{p}-\frac{1}{r}}\|f_1\|_{L^2}^\frac12\|f_2\|_{L^2}^\frac12
\end{equation}
holds for each $(p, r)\in \mathcal{N}(\tilde p,\tilde r)$. By rescaling, we obtain
$$\||e^{it\Delta}f_1 e^{it\Delta}f_2|^\frac12\|_{L_x^{p}L_t^{r}(\mathbb{R}^{n+1})}\lesssim 2^{j(\frac{n+1}{p}+\frac{1}{r}-\frac{n}{2})}\|f_1\|^\frac12_{L^2}\|f_2\|^\frac12_{L^2},$$
where $\hat{f_1},\ \hat{f_2}$ are supported on the subsets of $\mathbb{B}^n(\xi_0,2^{-j})\subset A(1)$ and $\text{dist}(\text{supp}\hat{f_1},\text{supp}\hat{f_2})\geq 2^{-j}$.
Theorem \ref{t1.1} is then a direct  result of the following crucial proposition.
\begin{proposition}\cite{lrv11}\label{p1.2}.
	Let $2<\tilde{p}<\tilde{r}<\infty$ and $\frac{n+1}{\tilde{p}}+\frac{1}{\tilde{r}}=\frac{n}{2}$. If 
	\begin{equation}\label{eq:2.1}
		\||e^{it\Delta}f_1 e^{it\Delta}f_2|^{\frac12}\|_{L_x^{p}L_t^{r}(\mathbb{R}^{n+1})}\lesssim 2^{j(\frac{n+1}{p}+\frac{1}{r}-\frac{n}{2})}\|f_1\|^\frac12_{L^2}\|f_2\|^\frac12_{L^2}
	\end{equation}
	holds for $(p,r)$ in a neighborhood of $(\tilde{p},\tilde{r})$,
	where $\hat{f_1},\ \hat{f_2}$ are supported on subsets of $\mathbb{B}^n(\xi_0,2^{-j})\subset A(1):=\{\xi:|\xi|\sim 1\}$ and $\text{dist}(\text{supp}\hat{f_1},\text{supp}\hat{f_2})\geq 2^{-j}$. Then, \eqref{S_{p,r}} holds for $(\tilde{p},\tilde{r})$.
\end{proposition}

Our aim becomes to find out the index $p_n$ such that \eqref{main estimate1} holds. For $n=2$, from Theorem \ref{t2.1}, \eqref{main estimate1} holds for $p_2=3$. For $n\geq 3$, by Theorem \ref{t4.6}, we have
$$\||e^{it\Delta}f_1e^{it\Delta}f_2|^{\frac12}\|_{L_x^{\tilde{p}_n}L_t^{\infty}(P_R)}\lesssim_\epsilon N^{\frac{1}{\tilde{p}_n}}R^{\epsilon}||f_1||_{L^2}^\frac12||f_2||_{L^2}^\frac12.$$
Interpolating the estimate with the local bilinear estimate of Tao \cite{t03},
$$\||e^{it\Delta}f_1e^{it\Delta}f_2|^{\frac12}\|_{L_x^{\frac{2(n+3)}{n+1}}L_t^{\frac{2(n+3)}{n+1}}(P_R)}\lesssim_\epsilon R^{\epsilon}||f_1||_{L^2}^\frac12||f_2||_{L^2}^\frac12,$$
 we obtain that \eqref{main estimate1} holds for
 $$p_n=\frac{2(n+1)[(n+2)\tilde{p}_n-2(n+3)]}{(n+1)n\tilde{p}_n-2(n^2+2n-1)}.$$
 
\section{The proof of Theorem \ref{c1.3}}\label{sub.2}

Let $\hat{f_1},\ \hat{f_2}$ be supported on subsets of $\mathbb{B}^n(0,1)$ with $\text{dist}(\text{supp}\hat{f_1},\text{supp}\hat{f_2})\geq \frac{1}{2}$. We make the wave packet decompositions at scale $R$,
$$e^{it\Delta}f_1=\sum_{\theta_1,\nu_1}e^{it\Delta}f_{\theta_1,\nu_1}\quad\text{and}\quad e^{it\Delta}f_2=\sum_{\theta_2,\nu_2}e^{it\Delta}f_{\theta_2,\nu_2},$$
where $  f_{\theta_i,\nu_i} (i=1,2)$ is the same as in \eqref{eq:2.2}.
The $R$-cube $Q_R$ can be decomposed into $O(R^{n\delta})$ many $R^{1-\delta}$-cubes $Q\in \mathcal{Q}_{R^{1-\delta}}$. 

The  relation "$\sim$" between the tubes   $T_{\theta_i,\nu_i}$ and the cubes $Q$ was introduced by Tao \cite{t03}. Here we state several facts about it without proof.  The readers can refer to \cite{t03} for the details.  Let $\mathbb T_i:=\cup T_{\theta_i,\nu_i},\,i=1,2$, and $\mathbb T=\mathbb T_1\cup \mathbb T_2.$ Since only the tubes intersect with $Q_R$ would have significant effects in the estimate, we assume that $\#\mathbb T\leq R^C$ for some constant $C$.
We decompose the $Q_R$ into $R^{\frac12}$ cubes $q$ and denote the sets as $\mathcal Q_{R^\frac12}$. For each $q\in \mathcal Q_{R^\frac12}$, set 
$$\mathbb{T}_i(q)=\big\{T_{\theta_i,\nu_i}\in\mathbb T_i;\,\,  T_{\theta_i,\nu_i}\cap R^\delta q\neq\emptyset \big\},\quad\text{for }\,\, i=1,2.$$ 
For some dyadic numbers $1\leq \mu_1,\mu_2\leq R^C$,   \begin{equation}
	q[\mu_1,\mu_2]=\big\{q\in\mathcal{Q}_{R^{\frac12}};\,\, \mu_i\leq \#\mathbb{T}_i(q)\leq 2\mu_i,\,\,i=1,2\big\},
\end{equation}
and for any fixed $T_{\theta_i,\nu_i}$
\begin{equation}
	\lambda(T_{\theta_i,\nu_i},\mu_1,\mu_2)=\sharp\Big\{q\in q[\mu_1,\mu_2];\,\,T_{\theta_i,\nu_i}\cap R^\delta q\neq\emptyset \Big\},\quad \text{for }\,\, i=1,2.
\end{equation}
It is easy to see that $1\leq \lambda(T_{\theta_i,\nu_i},\mu_1,\mu_2)\leq R^C$. Then for any dyadic number $1\leq \lambda_1,\lambda_2\leq R^C$, we introduce in the notation
\begin{equation}
	\mathbb T_i[\lambda_i,\mu_1,\mu_2]=\Big\{T_{\theta_i,\nu_i}\in\mathbb T_i;\,\,\lambda_i\leq \lambda(T_{\theta_i,\nu_i},\mu_1,\mu_2)\leq 2\lambda_i\Big\}\quad\text{for }\, i=1,2.
\end{equation}
So far, we have four dyadic parameters $\lambda_1,\lambda_2,\mu_1,\mu_2$, each has $O(\log R)$ many choices. Then for a fixed triple $(\lambda_1,\mu_1,\mu_2)$ and $T_{\theta_1,\nu_1}\in\mathbb T_1[\lambda_1,\mu_1,\mu_2]$, we say that $T_{\theta_1,\nu_1}\sim_{\lambda_1,\mu_1,\mu_2} Q$ if and only if the cube $Q\in\mathcal Q_{R^{1-\delta}}$ maximizes the quantity
$$\sharp\Big\{q\in q[\mu_1,\mu_2];\,\, T_{\theta_1,\nu_1}\cap R^\delta q\neq\emptyset,\,\,q\cap Q\neq\emptyset\Big\}.$$
By the definition, for any fixed  triple $(\lambda_1,\mu_1,\mu_2)$, there is $O(1)$ cubes $Q\in\mathcal{Q}_{R^{1-\delta}}$ satisfying $T_{\theta_1,\nu_1}\sim_{\lambda_1,\mu_1,\mu_2} Q$. And we call 
\begin{equation}
	T_{\theta_1,\nu_1}\sim Q\quad\text{if and only if}\quad\exists (\lambda_1,\mu_1,\mu_2)\quad\text{such that }T_{\theta_1,\nu_1}\sim_{\lambda_1,\mu_1,\mu_2} Q.
\end{equation} The relation $T_{\theta_2,\nu_2}\sim Q$ can be defined in the same way (see Figure 6). By the definition, it is easy to see that
\begin{equation}\label{e3.6}
	\sharp\big\{Q\in \mathcal{Q}_{R^{1-\delta}};\,\,T_{\theta_i,\nu_i}\sim Q\big\}\lesssim_{\epsilon}R^{\epsilon}\quad \text{for }i=1,2.
\end{equation}
The relation $T_{\theta_i,\nu_i}\sim Q$ describes the fact that the wave packet mainly concentrated in the cube $Q$. Otherwise, we use the notation 
$T_{\theta_i,\nu_i}\nsim Q$.  For this part, Tao proved the following proposition.
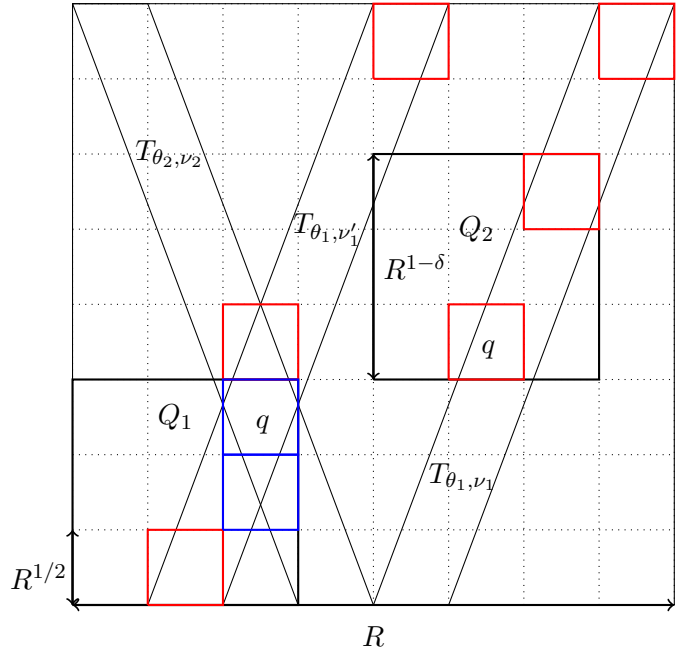
\begin{figure}
	
	\begin{tikzpicture}
		
		\draw  	(0,0) - - (0,8) - - (8,8) - - (8,0) - - (0,0);
		\draw[dotted,black] (0,0) grid (8,8);
		\draw[< - >,thick] (0,0) -- (8,0);
		\draw[fill] (4,-0.15) node[below] {$R$};
		\draw[< - >,thick] (4,3) -- (4,6);
		\draw[fill] (4,4.5) node[right] {$R^{1-\delta}$};
		\draw[< - >,thick] (0,0) -- (0,1);
		\draw[fill] (-0.45,0.7) node[below] {$R^{1/2}$};
		\draw[thick]  	(4,3) - - (4,6) - - (7,6) - - (7,3) - - (4,3);
		\draw[fill] (1,2.5) node[right] {$Q_1$};
		\draw[fill] (5,5) node[right] {$Q_2$};
		\draw[thick]	(0,0) - - (0,3) - - (3,3) - - (3,0) - - (0,0);
		\draw	(4,0) - - (7,8) - - (8,8) - - (5,0) - - (4,0);
		\draw[fill] (4.6,1.7) node[right] {$T_{\theta_1,\nu_1}$};
		\draw	(1,0) - - (4,8) - - (5,8) - - (2,0) - - (1,0);
		\draw[fill] (2.8,5) node[right] {$T_{\theta_1,\nu_1^{\prime}}$};
		\draw	(0,8) - - (1,8) - - (4,0) - - (3,0) - - (0,8);
		\draw[fill] (0.7,6) node[right] {$T_{\theta_2,\nu_2}$};
		\draw[thick,red]	(5,3) - - (5,4) - - (6,4) - - (6,3) - - (5,3);
		\draw[fill] (5.3,3.4) node[right] {$q$};
		\draw[fill] (2.3,2.4) node[right] {$q$};
		\draw[thick,red]	(6,5) - - (6,6) - - (7,6) - - (7,5) - - (6,5);
		\draw[thick,red]	(7,7) - - (7,8) - - (8,8) - - (8,7) - - (7,7);
		\draw[thick,red]	(1,0) - - (1,1) - - (2,1) - - (2,0) - - (1,0);
		\draw[thick,red]	(2,3) - - (2,4) - - (3,4) - - (3,3) - - (2,3);
		\draw[thick,red]	(4,7) - - (4,8) - - (5,8) - - (5,7) - - (4,7);
		\draw[thick,blue]	(2,1) - - (2,2) - - (3,2) - - (3,1) - - (2,1);
		\draw[thick,blue]	(2,2) - - (2,3) - - (3,3) - - (3,2) - - (2,2);
		
	\end{tikzpicture}
	\caption{$T_{\theta_1,\nu_1}\sim Q_2,\ T_{\theta_1,\nu_1^{\prime}}\sim Q_1,\ T_{\theta_2,\nu_2}\sim Q_1$}

\end{figure}
\begin{proposition}\cite{t03}\label{l2.2}
	For any given $T_{\theta_i,\nu_i}$, we have
	\begin{equation}\label{e2.7}
		\Big\|\Big|\Big(\sum_{T_{\theta_1,\nu_1}\nsim Q}e^{it\Delta}f_{\theta_1,\nu_1}\Big)\Big(\sum_{T_{\theta_2,\nu_2}\sim Q}e^{it\Delta}f_{\theta_2,\nu_2}\Big)\Big|^{\frac12}\Big\|_{L_{x,t}^4(Q)}\lesssim_\epsilon R^{\epsilon+c\delta-\frac{n-1}{8}}\|f_1\|^{\frac12}_{2}\|f_2\|^{\frac12}_{2},
	\end{equation}
	and
	\begin{equation}
		\Big\|\Big|\Big(\sum_{T_{\theta_1,\nu_1}\nsim Q}e^{it\Delta}f_{\theta_1,\nu_1}\Big)\Big(\sum_{T_{\theta_2,\nu_2}\nsim Q}e^{it\Delta}f_{\theta_2,\nu_2}\Big)\Big|^{\frac12}\Big\|_{L_{x,t}^4(Q)}\lesssim_\epsilon R^{\epsilon+c\delta-\frac{n-1}{8}}\|f_1\|^{\frac12}_{2}\|f_2\|^{\frac12}_{2}
	\end{equation}
	for some constant $c$ independent of $\delta,\epsilon$. 
\end{proposition}
By the Galilean transform $\mathcal{G}$, we have $\hat{\tilde{f_i}}(\xi)=\hat{f_i}(\xi-Ne_1)$  supported on subsets of $\mathbb{B}^n(Ne_1,1)$ with $\text{dist}(\text{supp}\hat{\tilde{f_1}},\text{supp}\hat{\tilde{f_2}})\geq \frac{1}{2}$. The cube $Q_R$ becomes the parallelepiped $P_R$, and the cube $Q\in\mathcal{Q}_{R^{1-\delta}}$ turns out to be the $R^{1-\delta}$-parallelepiped $\tilde{P}\in \mathcal{P}$ . The image of the tube $T_{\theta_i,\nu_i}$ will be denote as $\tilde{T}_{\theta_i,\nu_i}$. We define $\tilde{T}_{\theta_i,\nu_i}\sim \tilde{P}$ if and only if $T_{\theta_i,\nu_i}\sim Q$. By Proposition \ref{l2.2}, we have 
\begin{proposition}\label{l2.3}
	For any given $\tilde{T}_{\theta_i,\nu_i}$, 
	\begin{equation}\label{e2.15}
		\sharp\{\tilde{P}\in \mathcal{P}:\tilde{T}_{\theta_i,\nu_i}\sim \tilde{P}\}\lesssim_{\epsilon}R^{\epsilon},
	\end{equation}
	
	\begin{equation}\label{e2.14}
		\Big\|\Big|\Big(\sum_{\tilde{T}_{\theta_1,\nu_1}\nsim \tilde{P}}e^{i\tilde{t}\Delta}\tilde{f}_{\theta_1,\nu_1}\Big)\Big(\sum_{\tilde{T}_{\theta_2,\nu_2}\sim \tilde{P}}e^{i\tilde{t}\Delta}\tilde{f}_{\theta_2,\nu_2}\Big)\Big|^{\frac12}\Big\|_{L_{x,t}^4(\tilde{P})}\lesssim_\epsilon R^{\epsilon+c\delta-\frac{n-1}{8}}\|\tilde{f}_1\|^{\frac12}_{2}|\|\tilde{f}_2\|^{\frac12}_{2},
	\end{equation}
	and
	\begin{equation}
		\Big\|\Big|\Big(\sum_{\tilde{T}_{\theta_1,\nu_1}\nsim \tilde{P}}e^{i\tilde{t}\Delta}\tilde{f}_{\theta_1,\nu_1}\Big)\Big(\sum_{\tilde{T}_{\theta_2,\nu_2}\nsim \tilde{P}}e^{i\tilde{t}\Delta}\tilde{f}_{\theta_2,\nu_2}\Big)\Big|^{\frac12}\Big\|_{L_{x,t}^4(\tilde{P})}\lesssim_\epsilon R^{\epsilon+c\delta-\frac{n-1}{8}}\|\tilde{f}_1\|^{\frac12}_{2}|\|\tilde{f}_2\|^{\frac12}_{2}
	\end{equation}
	for some constant $c$ independent of $\delta,\epsilon$.
\end{proposition}

\begin{lemma}\label{l2.4}
	Let $p=3+\frac 17$, $r=5+\frac 12$ and $N\geq 1$. Suppose that $\hat{f_1},\ \hat{f_2}$ are supported on subsets of $\mathbb{B}^n(N e_1,1)$ with $\text{dist}(\text{supp}\hat{f_1},\text{supp}\hat{f_2})\geq \frac{1}{2}$. Then, for any $\epsilon>0$,
	\begin{equation}\label{e4.8}
		\||e^{it\Delta}f_1e^{it\Delta}f_2|^{\frac12}\|_{L_x^{p}L_t^{r}(P_R)}\leq R^{\epsilon}N^{\frac1p-\frac1r} \|f_1\|^{\frac12}_{L^2}\|f_2\|^{\frac12}_{L^2}.
	\end{equation}
	
\end{lemma}

\begin{proof}
	By Theorem \ref{t2.1} and H$\ddot{\text{o}}$lder's inequality, we have 
	\begin{equation}\label{e4.2}
		||e^{it\Delta}f||_{L_x^2L_t^{\infty}(P_R)}\leq (NR^2)^{\frac12-\frac13}||e^{it\Delta}f||_{L_x^3L_t^{\infty}(P_R)}\lesssim N^{\frac12}R^{\frac13+\epsilon}||f||_{L^2}.
	\end{equation}
	
	Let $C_R$ be the minimum number such that
	$$\||e^{it\Delta}f_1e^{it\Delta}f_2|^{\frac12}\|_{L_x^{p}L_t^{r}(P_R)}\leq C_RN^{\frac1p-\frac1r} \|f_1\|^{\frac12}_{L^2}\|f_2\|^{\frac12}_{L^2}.
	$$
	Assume that 
	\begin{equation}\label{e3.11}
		C_R\lesssim R^{\alpha}
	\end{equation}
	for some $\alpha>0$ and all $R\geq 1$. We claim that, for any $0<\delta\ll\epsilon\ll1$,  
	\begin{equation}\label{e3.12x}
		C_R\lesssim_{\delta,\epsilon}R^{\max\{(1-\delta)\alpha,c\delta\}+\epsilon},
	\end{equation}
	 where the constant $c$ is independent of $\delta,\epsilon,\alpha,R$. By a bootstrapping argument, we have
	$$C_R\lesssim_\epsilon R^\epsilon.$$ 
	
	We now give the proof of the claim. By H\"older's inequality and the bilinear Fourier restriction estimate, for $r>10/3$ 
	\begin{align*}
		\||e^{it\Delta}f_1e^{it\Delta}f_2|^{\frac12}\|_{L_x^{p}L_t^{r}(P_R)}
		& \leq (R^2N)^{\frac1p-\frac1r}\||e^{it\Delta}f_1e^{it\Delta}f_2|^{\frac12}\|_{L_x^{r}L_t^{r}(\mathbb{R}^{2+1})}\\
		& \lesssim (R^2N)^{\frac1p-\frac1r}\|f_1\|^{\frac12}_{L^2}\|f_2\|^{\frac12}_{L^2}. 	
	\end{align*}
	Hence \eqref{e3.11} holds for $\alpha= 2(\frac1p-\frac1r)$.
	
	$\||e^{it\Delta}f_1 e^{it\Delta}f_2|^{\frac12}\|_{L_x^pL_t^r(P_R)}$ can be decomposed into the  local part:
	\begin{equation}
		\sum_{\tilde{P}\in \mathcal{P}}\Big\|\Big|\Big(\sum_{T_{\theta_1,\nu_1}\sim \tilde{P}}e^{it\Delta}f_{\theta_1,\nu_1}\Big)\Big(\sum_{T_{\theta_2,\nu_2}\sim \tilde{P}}e^{it\Delta}f_{\theta_2,\nu_2}\Big)\Big|^{\frac12}\Big\|_{L_x^pL_t^r(\tilde{P})},
	\end{equation}
	and the global parts:
	\begin{align}
		&\sum_{\tilde{P}\in \mathcal{P}}\Big\|\Big|\Big(\sum_{T_{\theta_1,\nu_1}\nsim \tilde{P}}e^{it\Delta}f_{\theta_1,\nu_1}\Big)\Big(\sum_{T_{\theta_2,\nu_2}\sim \tilde{P}}e^{it\Delta}f_{\theta_2,\nu_2}\Big)\Big|^{\frac12}\Big\|_{L_x^pL_t^r(\tilde{P})},\label{eq:3.22x}\\
		&\sum_{\tilde{P}\in \mathcal{P}}\Big\|\Big|\Big(\sum_{T_{\theta_1,\nu_1}\sim \tilde{P}}e^{it\Delta}f_{\theta_1,\nu_1}\Big)\Big(\sum_{T_{\theta_2,\nu_2}\nsim \tilde{P}}e^{it\Delta}f_{\theta_2,\nu_2}\Big)\Big|^{\frac12}\Big\|_{L_x^pL_t^r(\tilde{P})},\label{eq:3.23x}\\
		&\sum_{\tilde{P}\in \mathcal{P}}\Big\|\Big|\Big(\sum_{T_{\theta_1,\nu_1}\nsim \tilde{P}}e^{it\Delta}f_{\theta_1,\nu_1}\Big)\Big(\sum_{T_{\theta_2,\nu_2}\nsim \tilde{P}}e^{it\Delta}f_{\theta_2,\nu_2}\Big)\Big|^{\frac12}\Big\|_{L_x^pL_t^r(\tilde{P})}.\label{eq:3.24x}
	\end{align}
	By Cauchy-Schwarz, \eqref{e2.15}, and \eqref{e3.11}, the local part can be bounded by
	\begin{align}
		&\sum_{\tilde{P}\in \mathcal{P}}\Big\|\Big|\Big(\sum_{T_{\theta_1,\nu_1}\sim \tilde{P}}e^{it\Delta}f_{\theta_1,\nu_1}\Big)\Big(\sum_{T_{\theta_2,\nu_2}\sim \tilde{P}}e^{it\Delta}f_{\theta_2,\nu_2}\Big)\Big|^{\frac12}\Big\|_{L_x^pL_t^r(\tilde{P})}\nonumber\\
		&\ \ \lesssim N^{\frac1p-\frac1r}\sum_{\tilde{P}\in \mathcal{P}} R^{(1-\delta)\alpha} \Big\|\sum_{T_{\theta_1,\nu_1}\sim \tilde{P}} f_{\theta_1,\nu_1} \Big\|^{\frac12}_{L^2} \Big\|\sum_{T_{\theta_2,\nu_2}\sim \tilde{P}} f_{\theta_2,\nu_2} \Big\|^{\frac12}_{L^2}\nonumber\\
		&\ \  \lesssim N^{\frac1p-\frac1r} R^{(1-\delta)\alpha} \Big(\sum_{\tilde{P}\in \mathcal{P}}\Big\|\sum_{T_{\theta_1,\nu_1}\sim \tilde{P}}f_{\theta_1,\nu_1} \Big\|_{L^2}^2\Big)^{\frac{1}{4}}\Big(\sum_{\tilde{P}\in \mathcal{P}}\Big\|\sum_{T_{\theta_2,\nu_2}\sim \tilde{P}}f_{\theta_2,\nu_2} \Big\|_{L^2}^2\Big)^{\frac{1}{4}}\nonumber \\
		&\ \  \lesssim N^{\frac1p-\frac1r} R^{(1-\delta)\alpha} \Big(\sum_{T_{\theta_1,\nu_1}}\sum_{T_{\theta_1,\nu_1}\sim \tilde{P}}\|f_{\theta_1,\nu_1}\|_{L^2}^2\Big)^{\frac{1}{4}}\Big(\sum_{T_{\theta_2,\nu_2}}\sum_{T_{\theta_2,\nu_2}\sim \tilde{P}}\|f_{\theta_2,\nu_2}\|_{L^2}^2\Big)^{\frac{1}{4}}\nonumber \\
		&\ \  \lesssim_\epsilon N^{\frac1p-\frac1r}R^{(1-\delta)\alpha+\epsilon} \|f_1\|^{\frac12}_{L^2}\|f_2\|^{\frac12}_{L^2}.\label{e3.9x}
	\end{align}
	For the global parts, we give the estimate of \eqref{eq:3.22x}. \eqref{eq:3.23x} and \eqref{eq:3.24x} can be bounded in the same way. From Cauchy-Schwarz and \eqref{e4.2}, it follows that
	\begin{align}\label{e2.16}
		&\Big\|\Big|\Big(\sum_{T_{\theta_1,\nu_1}\nsim \tilde{P}}e^{it\Delta}f_{\theta_1,\nu_1}\Big)\Big(\sum_{T_{\theta_2,\nu_2}\sim \tilde{P}}e^{it\Delta}f_{\theta_2,\nu_2}\Big)\Big|^{\frac12}\Big\|_{L_x^2L_t^{\infty}(\tilde{P})}\nonumber\\
		&\ \ \leq  \Big\|\sum_{T_{\theta_1,\nu_1}\nsim \tilde{P}}e^{it\Delta}f_{\theta_1,\nu_1}\Big\|^{\frac12}_{L_x^2L_t^{\infty}(\tilde{P})}\Big\|\sum_{T_{\theta_2,\nu_2}\sim \tilde{P}}e^{it\Delta}f_{\theta_2,\nu_2}\Big\|^{\frac12}_{L_x^2L_t^{\infty}(\tilde{P})}\nonumber\\
		&\ \  \lesssim_\epsilon N^{\frac12}R^{\frac13+\epsilon} \|f_1\|^{\frac12}_{L^2}\|f_2\|^{\frac12}_{L^2}.
	\end{align}  
	Interpolating  \eqref{e2.16} with \eqref{e2.14} $(n=2)$, we obtain 
	$$\Big\|\Big|\Big(\sum_{T_{\theta_1,\nu_1}\nsim \tilde{P}}e^{it\Delta}f_{\theta_1,\nu_1}\Big)\Big(\sum_{T_{\theta_2,\nu_2}\sim \tilde{P}}e^{it\Delta}f_{\theta_2,\nu_2}\Big)\Big|^{\frac12}\Big\|_{L_x^pL_t^r(\tilde{P})}\lesssim_\epsilon N^{\frac1p-\frac1r}R^{c\delta+\epsilon} \|f_1\|^{\frac12}_{L^2}\|f_2\|^{\frac12}_{L^2}$$  for  $p=22/7$ and $r=22/4$.
	Notice that $\sharp\{\tilde{P}:\tilde{P}\in\mathcal{P}\}\lesssim R^{(n+1)\delta}$, 
	$$
	\sum_{\tilde{P}\in \mathcal{P}}\Big\|\Big|\Big(\sum_{T_{\theta_1,\nu_1}\nsim \tilde{P}}e^{it\Delta}f_{\theta_1,\nu_1}\Big)\Big(\sum_{T_{\theta_2,\nu_2}\sim \tilde{P}}e^{it\Delta}f_{\theta_2,\nu_2}\Big)\Big|^{\frac12}\Big\|_{L_x^{p}L_t^{r}(\tilde{P})}\lesssim_{\epsilon} N^{\frac1p-\frac1r}R^{c\delta+\epsilon} \|f_1\|^{\frac12}_{L^2}\|f_2\|^{\frac12}_{L^2}.
	$$
	Combining this estimate and \eqref{e3.9x}, we finally obtain
	$$C_R\lesssim_{\delta,\epsilon}R^{\max\{(1-\delta)\alpha,c\delta\}+\epsilon}.$$	
\end{proof}

From \eqref{e4.8} and Proposition \ref{p2.2}, we have
$$\sup_{\rho>1}\||e^{it\Delta}f_1e^{it\Delta}f_2|^{\frac12}\|_{L_x^{p_0}L_t^{r_0}(\mathbb{R}^2\times[0,\rho])}\lesssim N^{\frac{1}{p_0}-\frac{1}{r_0}} \|f_1\|^{\frac12}_{L^2}\|f_2\|^{\frac12}_{L^2}$$
 for $p_0>3+\frac17$ and $r_0>5+\frac12$. By the following Proposition \ref{p2.1}, we obtain that
$$\|e^{it\Delta}f\|_{L_x^pL_t^r(\mathbb{R}^2\times[0,1])}\lesssim\|f\|_{B_{\alpha,p}^p}$$
holds for $\alpha=n(1-2/p)-2/r$, $p>3+\frac17$ and $r\geq5+\frac12$.

\begin{proposition}\cite{lrs13}\label{p2.1}
	Let $2(n+1)/n\leq p_0\leq r_0\leq 2p_0/(p_0-2)$ and $N\geq 1$. Suppose the Fourier supports of $f_1, f_2$ are in $\mathbb{B}^n(Ne_1, 1)$ with $dist(supp\hat{f_1},supp\hat{f}_2)\geq \frac12$ and
	\begin{equation*}
		\||e^{it\Delta}f_1e^{it\Delta}f_2|^{\frac12}\|_{L_x^{p_0}L_t^{r_0}(\mathbb{R}^n\times[0,\rho])}\lesssim N^{\gamma} \|f_1\|_{L^2}^{\frac12}\|f_2\|_{L^2}^{\frac12}
	\end{equation*}
holds for some $\gamma<2n(1-\frac{2}{p_0})-\frac{4}{r_0}$ and any $\rho>1$. Then, for any $p_0<p<\infty$ and $r_0\leq r<\infty$,
	$$\|e^{it\Delta}f\|_{L_x^pL_t^r(\mathbb{R}^n\times[0,1])}\lesssim\|f\|_{B_{\alpha,p}^p}$$
	holds for $\alpha=n(1-2/p)-2/r$.
\end{proposition}

\appendix
\section{Epsilon removal for bilinear space-time estimate}

\begin{proposition}\label{p2.2}
Let $N\geq 1$. Suppose the Fourier supports of $f_1, f_2$ are in $\mathbb{B}^n(Ne_1, 1)$ with $$dist(supp\hat{f_1},supp\hat{f}_2)\geq \frac12$$ and
	\begin{equation}\label{n bilinear estimate}
		\||e^{it\Delta}f_1e^{it\Delta}f_2|^{\frac12}\|_{L_x^{p_0}L_t^{r_0}(P_R)}\lesssim_\epsilon N^{\frac{1}{p_0}-\frac{1}{r_0}}R^{\epsilon} \|f_1\|_{L^2}^{\frac12}\|f_2\|_{L^2}^{\frac12}
	\end{equation}
holds for $r_0\geq p_0\geq2$.
	Then, for any $p>p_0$ and $r>r_0$,
	$$\||e^{it\Delta}f_1e^{it\Delta}f_2|^{\frac12}\|_{L_x^{p_0}L_t^{r_0}(\mathbb{R}^{n+1})}\lesssim N^{\frac{1}{p}-\frac{1}{r}} \|f_1\|_{L^2}^{\frac12}\|f_2\|_{L^2}^{\frac12}.$$
\end{proposition}

The above space-time $\epsilon$-removing argument was proven by Lee, Rogers and Vargas \cite{lrv11}. Here, we give another proof by making use of an argument of Bourgain and Guth \cite{bg11}.

\begin{definition}[$(M,R)$-sparse parallelepipeds]
 A finite collection $\{P_R(z_{\alpha})\}_{\alpha=1}^{M}$ of parallelepipeds in $\mathbb{R}^{n+1}$ is called $(M,R)$-sparse 
	if the centers $z_{\alpha}$ are $(MR)^{\gamma}$-separated, i.e.
	$$|z_{\alpha_0}-z_{\alpha_1}|\geq (MR)^{\gamma}\ \ \text{with}\ \ \alpha_0\neq \alpha_1,$$
	where $\gamma$ is a fixed number.
\end{definition}

The following cover lemma which was first proved by Tao \cite[Lemma 3.3]{t99} and refined by Cho, Koh and Lee \cite[Lemma 2.6]{ckl22}. We transform their result into the parallelepiped by Galilean transform.

\begin{lemma}\label{l4.1}
	Let $E\subset \mathbb R^{n+1}$ be  a finite union of finitely overlapping parallelepipeds of size $c\sim 1$. Then for each $K\in \mathbb{N}$, there are subsets $E_1,E_2,...,E_K$ of $E$ with
	$$E=\bigcup_{k=1}^{K}E_k$$
	such that each $E_k$ has $O(|E|^{1/K})$ number of $(O(|E|),|E|^{O({\gamma}^{k-1})})$ sparse collections
	$$S_1,S_2,...,S_{O(E^{1/K})}$$
	of which the union $S_1\cup S_2\cup...\cup S_{O(E^{1/K})}$ is a covering of $E_k$.
\end{lemma}

\begin{proof}
	We assume that $\|f_i\|_{L^2}=1\,(i=1,2)$ and denote $F(x,t):=|e^{it\Delta}f_1(x)e^{it\Delta}f_2(x)|^{\frac12} $. Given a $(M,R)$-sparse set $E=\bigcup_{\alpha=1}^{M}P_R(z_{\alpha})$ with $z_{\alpha}=(x_{\alpha},t_{\alpha})$. We first  show that
	\begin{equation}\label{e4.1}
		\|F|_{E}\|_{L_x^{p_0}L_t^{r_0}}\leq C_{\epsilon}N^{\frac{1}{p_0}-\frac{1}{r_0}}R^{\epsilon}.
	\end{equation}
	
	In fact, let $\phi (x,t;\xi)=x\cdot\xi-2\pi t|\xi|^{2}$ and $(x,t)\in P_R(z_{\alpha})$, we write
	\begin{equation}\label{e4.2x}
		e^{it\Delta}f(x)=\int e^{2\pi i (\phi (x,t;\xi)-\phi (x_{\alpha},t_{\alpha};\xi))}(e^{2\pi i \phi (x_{\alpha},t_{\alpha};\xi)}\hat{f}(\xi))\omega(\xi) d\xi,
	\end{equation}
	where $\omega(\xi)$ is a smooth function localized on $\supp \hat{f}$. Define $\hat{g}(\xi):=e^{2\pi i \phi (x_{\alpha},t_{\alpha};\xi)}\hat{f}(\xi)$, then
	\begin{equation}\label{e4.3x}
		\eqref{e4.2x}=\int \int e^{2\pi i (\phi (x,t;\xi)-\phi (x_{\alpha},t_{\alpha};\xi)-\xi y)}\omega(\xi) d\xi g(y)dy.
	\end{equation}
	Let $\eta$ be a bump function satisfy $0\leq \eta \leq 1$ and $\supp \eta\subset \mathbb{B}^n(0,1)$, $\eta_{R_1}(y):=\eta(\frac{y}{R_1})$, $R_1:=100MR$ and $P_{(0,R_1)}g:=g\eta_{R_1}$. Since 
	\begin{align*}
		|\bigtriangledown_{\xi}[\phi (x,t;\xi)-\phi (x_{\alpha},t_{\alpha};\xi)]|
		&=|x-x_{\alpha}-4\pi(t-t_{\alpha})\xi|\\
		& \leq |x-x_{\alpha}-4\pi(t-t_{\alpha})Ne_1|+|4\pi(t-t_{\alpha})(\xi-Ne_1)|\\
		& \lesssim R,
	\end{align*}
	we have 
	$$|\eqref{e4.3x}|\lesssim \lf|\int \int e^{2\pi i (\phi (x,t;\xi)-\phi (x_{\alpha},t_{\alpha};\xi)-\xi y)}\omega(\xi) d\xi P_{(0,R_1)}g(y)dy\r|+\text{RapDec}(MR).$$ Define $$\hat{f_{\alpha}}:=e^{-2\pi i\phi (x_{\alpha},t_{\alpha};\xi)}\widehat{P_{(0,R_1)}g},$$ then 
	$$|e^{it\Delta}f(x)|\lesssim |e^{it\Delta}f_{\alpha}(x)|+\text{RapDec}(MR).$$ 	
 By the above argument, $r_0\geq p_0\geq2$ and \eqref{n bilinear estimate}, we have
	\begin{align}\label{e4.3}
		\|F|_E\|_{L_x^{p_0}L_t^{r_0}}^{p_0}
& \leq \sum_{\alpha}\|F\|_{L_x^{p_0}L_t^{r_0}(P_R(z_{\alpha}))}^{p_0}\nonumber\\ 
& \lesssim_{\epsilon}R^{p_0\epsilon}N^{p_0(\frac{1}{p_0}-\frac{1}{r_0})}\sum_{\alpha}(\|f_{1,\alpha}\|_{L^2}^{\frac12} \|f_{2,\alpha}\|_{L^2}^{\frac12})^{p_0}\nonumber\\
	&	\lesssim_{\epsilon}R^{p_0\epsilon}N^{p_0(\frac{1}{p_0}-\frac{1}{r_0})}\Big(\sum_{\alpha}\|f_{\alpha}\|_{L^2}^2\Big)^{\frac{p_0}{2}},
	\end{align}
	where $$\sum_{\alpha}\|f_{\alpha}\|_{L^2}^2=\max\Big\{\sum_{\alpha}\|f_{1,\alpha}\|_{L^2}^2,\sum_{\alpha}\|f_{2,\alpha}\|_{L^2}^2 \Big\}.$$
By duality and Plancherel's theorem, there exist $\{\zeta_{\alpha}\}_{\alpha}$ satisfying supp $\zeta_{\alpha}\subset \mathbb{B}^n(0,R_1)$ and $\sum_{\alpha}\|\zeta_{\alpha}\|_{L^2}^2=1$ such that

\begin{align}\label{eA.3}
	\Big(\sum_{\alpha}\|f_{\alpha}\|_{L^2}^2\Big)^{\frac{1}{2}} 
	&  
	=\sum_{\alpha}<P_{(0,R_1)}g,\zeta_{\alpha}> \nonumber\\
	& 
	=\sum_{\alpha}<e^{2\pi i \phi (x_{\alpha},t_{\alpha};\xi)}\hat{f},\hat{\zeta_{\alpha}}> \nonumber\\
	& 
	=\sum_{\alpha}<\hat{f},e^{2\pi i \phi (x_{\alpha},t_{\alpha};\xi)}\hat{\zeta_{\alpha}}>.
\end{align} 
From \eqref{eA.3} and Cauchy-Schwarz inequality, it follows that
\begin{align}\label{eA.4}
	\Big(\sum_{\alpha}\|f_{\alpha}\|_{L^2}^2\Big)^{\frac{1}{2}}
	& \leq \Big\|\sum_{\alpha}e^{2\pi i \phi (x_{\alpha},t_{\alpha};\xi)}\hat{\zeta_{\alpha}}\Big\|_{L^2}\nonumber\\ 
	& \leq \Big(\sum_{\alpha}\|\zeta_{\alpha}\|_{L^2}^2\Big)^{\frac12}+\Big(\sum_{\alpha,\beta:\alpha\neq\beta}  |<e^{2\pi i \phi (x_{\alpha},t_{\alpha};\xi)}\hat{\zeta_{\alpha}},e^{2\pi i \phi (x_{\beta},t_{\beta};\xi)}\hat{\zeta_{\beta}}>|\Big)^{\frac12}.
\end{align}
For the off-diagonal terms $\alpha\neq \beta$,
\begin{align}\label{eA.5}
		 |<e^{2\pi i \phi (x_{\alpha},t_{\alpha};\xi)}\hat{\zeta_{\alpha}},e^{2\pi i \phi (x_{\beta},t_{\beta};\xi)}\hat{\zeta_{\beta}}>|
		 =\Big|\int e^{2\pi i [\phi (x_{\alpha},t_{\alpha};\xi)-\phi (x_{\beta},t_{\beta};\xi)]}\hat{\zeta_{\alpha}}(\xi)\overline{\hat{\zeta_{\beta}}}(\xi)d\xi\Big|.
\end{align}
Since $|z_{\alpha}-z_{\beta}|\geq (MR)^{\gamma}$ with $z_{\alpha}=(x_{\alpha},t_{\alpha}),\ z_{\beta}=(x_{\beta},t_{\beta})$, $|\phi (x_{\alpha},t_{\alpha};\xi)-\phi (x_{\beta},t_{\beta};\xi)|$ satisfies either
$$|\bigtriangledown_{\xi}[\phi (x_{\alpha},t_{\alpha};\xi)-\phi (x_{\beta},t_{\beta};\xi)]|\gtrsim (MR)^{\gamma}$$
or
$$|\text{det}\, D_{\xi}^2 [\phi (x_{\alpha},t_{\alpha};\xi)-\phi (x_{\beta},t_{\beta};\xi)]|\gtrsim (MR)^{n\gamma}.$$
By the method of stationary phase, we have 
$$\text{\eqref{eA.5}}\lesssim (MR)^{-n\gamma/2}\|\zeta_{\alpha}\|_{L^1}\|\zeta_{\beta}\|_{L^1}\lesssim R_1^n(MR)^{-n\gamma/2}\|\zeta_{\alpha}\|_{L^2}\|\zeta_{\beta}\|_{L^2}.$$
By choosing $\gamma>2$, we obtain
$\eqref{eA.4} \lesssim 1$ which implies \eqref{e4.1}.

	Since $\hat{f}_1$ and $\hat{f}_2$ are supported on $\mathbb{B}^n(Ne_1, 1)$, we have $\|F\|_\infty\leq \|f_1\|_{L^2}\|f_2\|_{L^2}=1$. We decompose 
	$$F=\sum_{k\geq 0}F|_{|F|\sim 2^{-k}}:=\sum_{k\geq 0}F_{k},$$
	where $\text{supp}F_k\subset E_k$ and $E_k$ is the union of parallelepipeds of size $c$ ($c$-parallelepipeds). We denote by proj$(E_k)$ the projection of $E_k$ along the $(4\pi Ne_1,1)$ onto the $x$-plane. For each grid point $x\in c\mathbb{Z}^n\cap \text{proj}(E_k) $, we define $E_{k,x}$ to be the union of
	$c$-parallelepipeds such that $x\in$proj$(E_k)$ and
	$$E_k^j=\bigcup\{E_{k,x}: \sharp(E_{k,x})\sim 2^j\}.$$
 We have $$E_k=\bigcup _{j\geq 0}E_k^j.$$
	Let $$I=\frac{\log(1/\epsilon)}{2\log \gamma}+1.$$ By Lemma \ref{l4.1}, we have 
	$$E_k^j=\bigcup_{i= 1}^IE_{k,i}^j\ \ \text{and}\ \ E_{k,i}^j\subset \bigcup_{|m|\lesssim|E_k^j|^{1/I}}\bigcup_{\alpha\lesssim |E_k^j|}P_{R_m}(z_{\alpha}),$$
	where $R_m\lesssim |E_k^j|^{C\gamma^{i-1}}$. From \eqref{e4.1}, we obtain that 
	$$\|F_k|_{E_{k,i}^j}\|_{L_x^{p_0}L_t^{r_0}}\leq C_{\epsilon} N^{\frac{1}{p_0}-\frac{1}{r_0}}|E_{k}^{j}|^{1/I}(|E_k^j|^{C{\gamma}^{i-1}})^{\epsilon}.$$
	Sum up $1\leq i\leq I$,
	$$\|F_k|_{E_{k}^j}\|_{L_x^{p_0}L_t^{r_0}}\leq C_{\epsilon}N^{\frac{1}{p_0}-\frac{1}{r_0}}|E_{k}^{j}|^{1/I}(|E_k^j|^{C{\gamma}^{I-1}})^{\epsilon}.$$
	Combining this estimate and the definition of $E_{k}^{j}$, we have
	\begin{equation*}
		2^{-k}2^{j/r_0}|\text{proj}(E_k^j)|^{1/p_0}\lesssim N^{\frac{1}{p_0}-\frac{1}{r_0}} (2^j|\text{proj}(E_k^j)|)^{\delta(\epsilon)},
	\end{equation*}
	where $\delta(\epsilon)=1/I+C\epsilon{\gamma}^{I-1}$. Since $\lim_{\epsilon \to 0}\delta(\epsilon)=0$, then there exists $\epsilon>0$ such that 
	$$\frac{1}{r}+\epsilon\leq \Big(\frac{1}{r_0}-\delta(\epsilon)\Big)(1-\epsilon) $$
	and
	$$\frac{1}{p_1}=  \Big(\frac{1}{p_0}-\delta(\epsilon)\Big)(1-\epsilon).$$
	For the decomposition of $F$, we have 
	\begin{align}\label{ea.1}
		\|F\|_{L_x^{p_1}L_t^{r}}
		& \leq \sum_{k\geq 0}\sum_{j\geq 0}\|F_k|_{E_{k}^j}\|_{L_x^{p_1}L_t^{r}} \nonumber\\
		& \lesssim \sum_{k\geq 0}\sum_{j\geq 0}2^{-k}2^{j/r}|\text{proj}(E_k^j)|^{1/p_1}\nonumber\\
		& \lesssim N^{\frac{1}{p_1}-\frac{1}{r}}\sum_{k\geq 0}\sum_{j\geq 0}2^{(-k-j)\epsilon}\nonumber\\
		& \lesssim CN^{\frac{1}{p_1}-\frac{1}{r}}.
	\end{align}
	On the other hand, for $1\leq r\leq\infty$,
	$$\||e^{it\Delta}f_1e^{it\Delta}f_2|^[\frac12]\|_{L_x^{\infty}L_t^{r}(\mathbb{R}^{n+1})}\lesssim N^{-\frac1r}\|f_1\|^{\frac12}_{L^2}\|f_2\|^{\frac12}_{L^2}.$$
	It was proved by \cite{r09}. By interpolation it with \eqref{ea.1}, we obtain for $p\geq p_1$,
	$$\||e^{it\Delta}f_1e^{it\Delta}f_2|^{\frac12}\|_{L_x^{p}L_t^{r}(\mathbb{R}^{n+1})} \leq C_{p,r} N^{\frac1p-\frac1r}\|f_1\|^{\frac12}_{L^2}\|f_2\|^{\frac12}_{L^2}.$$
	Hence, we complete the proof of Proposition \ref{p2.2}.
	
\end{proof}




\section{A bilinear space-time estimate conjecture}\label{s5}

We now formulate a bilinear space-time conjecture which implies Conjecture \ref{c1.1}.

\begin{conjecture}[Bilinear space-time estimate]\label{c1.2} Suppose that $\hat{f_1},\ \hat{f_2}$ are supported on subsets of $\mathbb{B}^n(N e_1,1)$ with $\text{dist}(\text{supp}\hat{f_1},\text{supp}\hat{f_2})\geq \frac{1}{2}$.
	\begin{equation}\label{e4.1a}
		\||e^{it\Delta}f_1e^{it\Delta}f_2|^{\frac12}\|_{L_x^{p}L_t^{r}(\mathbb{R}^{n+1})}\leq C_{p,r} N^{\frac{1}{p}-\frac{1}{r}} \|f_1\|^{\frac12}_{L^2}\|f_2\|^{\frac12}_{L^2},
	\end{equation} holds if and only if  $\frac{n+2}{p}+\frac{1}{r}\leq \frac{n+1}{2}, 2\leq p,r<\infty$. 
\end{conjecture}

\begin{remark}
	\begin{itemize}
		\item[\rm{(I)}] By the same argument of Subsection \ref{sub.1} of this paper, we have Conjecture \ref{c1.2} $\implies$ Conjecture \ref{c1.1}.
		\item[\rm{(II)}] Conjecture \ref{c1.2} is also closely related to Conjecture \ref{c1.4}. By Proposition \ref{p2.1}, we have
	$$\text{Conjecture \ref{c1.2}}\ \implies\ \text{Conjecture \ref{c1.4} holds for}\ p>\frac{2(n+1)}{n}\ \text{and}\ r\geq 2(n+1) .$$
	The above conjecture is also deeply related to the Fourier restrictive conjecture of paraboloid in $\mathbb{R}^{n+1}$ (see \cite[Theorem 1.1]{lrs13}). However, it is very difficult to prove the restrictive conjecture with the endpoint $\frac{2(n+1)}{n}$. Conjecture \ref{c1.2} provides a new viewpoint to improve Conjecture \ref{c1.4}.
	\end{itemize}
\end{remark}

We now give a necessary condition for \eqref{e4.1a} by modifying an example from  Foschi-Klainerman \cite{fk00}. 
	 
	 \begin{proposition}
	 	Let $p,r\in [2,\infty)$. If \eqref{e4.1a} is true, then
	 	$$\frac{n+2}{p}+\frac{1}{r}\leq \frac{n+1}{2}.$$
	 \end{proposition}
	 \begin{proof}
	 	Let $\epsilon>0$,  $$\hat{f_1}(\xi)=\chi_{\{|\xi_1-(N+1/2)|<\epsilon\}}(\xi_1)\chi_{\{|\xi_2-1/2|<\epsilon^2\}}(\xi_2)\chi_{\{|\xi_3|<\epsilon\}}(\xi_3)...\chi_{\{|\xi_n|<\epsilon\}}(\xi_n)$$
	 	and
	 	$$\hat{f_2}(\eta)=\chi_{\{|\eta_1-(N+1/2)|<\epsilon\}}(\eta_1)\chi_{\{|\eta_2+1/2|<\epsilon^2\}}(\eta_2)\chi_{\{|\eta_3|<\epsilon\}}(\eta_3)...\chi_{\{|\eta_n|<\epsilon\}}(\eta_n).$$
	 	Then $$\text{supp}\hat{f_i}\subset \mathbb B^n(Ne_1,1),\quad \|f_i\|_{L^2}\sim \epsilon^{\frac{n+1}{2}},\quad \text{for } i=1,2.$$
	 	The distance between their Fourier supports is at least $\frac12$. 
	 	An easy calculus shows that
	 	\begin{align*}
	 		|\xi|^2
	 		& =|(\xi_1-(N+1/2))+(N+1/2)|^2+|(\xi_2-1/2)+1/2|^2+|\xi_3|^2+...+|\xi_n|^2\\
	 		& \sim N^2+2(N+1/2)(\xi_1-1)+O(\epsilon^2) 	
	 	\end{align*} 
	 	and  
	 	$$|\eta|^2\sim N^2+2(N+1/2)(\eta_1-1)+O(\epsilon^2).$$
	 	Let $A\subset\mathbb R^{n+1}$ denotes the set 	
$$
	x_1 -4\pi (N+1/2)t=O(\epsilon^{-1}),\ x_2=O(\epsilon^{-2}),
	x_i =O(\epsilon^{-1}),3\leq i\leq n, t=O(\epsilon^{-2}).
 $$
	 	 	
	 	For each point $(x,t)\in A$, we have
	 	\begin{align*}
	 		|e^{it\Delta}f_1e^{it\Delta}f_2(x)|
	 		& =\lf|\int_{\mathbb{R}^n\times \mathbb{R}^n}e^{2\pi i [x\cdot(\xi+\eta)-2\pi t(|\xi|^2+|\eta|^2)]}\hat{f_1}(\xi)\hat{f_2}(\eta)d\xi d\eta\r|\\
	 		& \sim \epsilon^{2(n+1)}.
	 	\end{align*}
	 	Hence 
	 	\begin{equation}
	 		\||e^{it\Delta}f_1e^{it\Delta}f_2(x)|^{\frac12}\|_{L^p_xL^r_t(\mathbb{R}^n\times \mathbb{R}^n)}\geq \epsilon^{n+1}\|\chi_A\|_{L^p_xL^r_t(\mathbb{R}^n\times \mathbb{R}^n)}\geq \epsilon^{(n+1)-\frac{1}{r}-\frac{n+2}{p}}N^{\frac1p-\frac1r}.
	 	\end{equation}	
	 	From the estimate \eqref{e4.1a}, we have 
	 	$$\epsilon^{(n+1)-\frac{1}{r}-\frac{n+2}{p}}N^{\frac1p-\frac1r}\lesssim N^{\frac1p-\frac1r}\epsilon^{\frac{n+1}{2}}.$$
	 It implies that $\frac{n+2}{p}+\frac{1}{r}\leq \frac{n+1}{2}$.
	 \end{proof}

\subsection*{Acknowledgement}

The authors would like to express their gratitude to Dr. Zhenbin Cao for many valuable comments and suggestions on the early visions of this paper.




\end{document}